\documentclass[11pt, twoside]{article}

\usepackage{amsmath}
\usepackage{amssymb}
\usepackage{titletoc}
\usepackage{mathrsfs}
\usepackage{amsthm}

\usepackage{indentfirst}
\usepackage{color}
\usepackage{txfonts}
\usepackage{enumerate}
\usepackage{anysize}

\allowdisplaybreaks

\newtheorem{theorem}{Theorem}[section]
\newtheorem{lemma}[theorem]{Lemma}
\newtheorem{corollary}[theorem]{Corollary}
\newtheorem{proposition}[theorem]{Proposition}
\theoremstyle{definition}

\newtheorem{remark}[theorem]{Remark}
\newtheorem{definition}[theorem]{Definition}

\numberwithin{equation}{section}

\begin{document}

\title{\bf\Large
Matrix-Weighted Campanato Spaces:
Duality and Calder\'on--Zygmund Operators
\footnotetext{\hspace{-0.35cm} 2020 {\it Mathematics Subject Classification}.
Primary 46E35; Secondary 42B20, 42B30, 42B35, 46E40, 47A56.\endgraf
{\it Key words and phrases.} matrix weight, Campanato space,
Hardy space, reducing operator, duality,
Calder\'on--Zygmund operator.\endgraf
This project is partially supported by
the National Natural Science Foundation of China
(Grant Nos. 12431006 and 12371093),
and the Fundamental Research Funds
for the Central Universities (Grant Nos. 2253200028 and 2233300008).}}
\date{}
\author{Yiqun Chen, Dachun Yang\footnote{Corresponding author,
E-mail: \texttt{dcyang@bnu.edu.cn}/{\color{red}\today}/Final Version.}\
\ and Wen Yuan}

\maketitle

\vspace{-0.8cm}

\begin{center}
\begin{minipage}{13.8cm}
{\small {\bf Abstract}\quad
Let $p\in(0,\infty)$, $q\in[1,\infty)$, $s\in\mathbb Z_+$,
and $W$ be an $A_p$-matrix weight,
which in the scalar case is exactly a
Muckenhoupt $A_{\max\{1,p\}}$ weight.
In this article, by using the
reducing operators of $W$, we introduce matrix-weighted
Campanato spaces $\mathcal L_{p,q,s,W}$. When $p\in(0,1]$,
applying the atomic and
the finite atomic characterizations of the matrix-weighted Hardy
space $H^p_W$, we prove that the dual space of $H^p_W$
is precisely $\mathcal L_{p,q,s,W}$, which further induces
several equivalent characterizations
of $\mathcal L_{p,q,s,W}$. In addition, we
obtain a necessary and sufficient condition for the boundedness
of  modified Calder\'on--Zygmund operators on $\mathcal L_{p,q,s,W}$
with $p\in(0,\infty)$,
which, combined with the duality, further gives a necessary
and sufficient condition for the boundedness of Calder\'on--Zygmund
operators on $H^p_W$ with $p\in(0,1]$.}
\end{minipage}
\end{center}

\vspace{0.2cm}



\section{Introduction}
The study of matrix-weighted function spaces and the related
boundedness of operators can be traced to the investigation of
Wiener and Masani \cite{wm58} on the prediction theory of
multivariate stochastic processes, they
introduced the matrix-weighted Lebesgue space
$L^2_W$ on $\mathbb{R}^n$, where $W$ is a matrix weight.
Since the 1990s,
motivated by problems concerning the angle between the
past and the future of the multivariate random stationary process,
as well as the boundedness of the inverse of Toeplitz operators,
this field  has witnessed remarkable advancements.
Treil and Volberg \cite{tv97} introduced matrix
$A_2$ weights on $\mathbb R$ and proved that
the Hilbert transform is bounded on $L^2_W$ over $\mathbb R$
if and only if $W$ is such
an $A_2$ matrix weight.
Subsequently, Nazarov and Treil \cite{nt96} and Volberg \cite{v97}
independently extended this result to the matrix-weighted Lebesgue space
$L^p_W$ on $\mathbb R$ with $W\in A_p$ for the full range $p\in(1,\infty)$.
Christ and Goldberg \cite{cg01} introduced the matrix-weighted
maximal operator, and Goldberg \cite{Gold} subsequently proved
its boundedness on $L^p_W$ spaces as well as the boundedness of
certain convolutional Calder\'on--Zygmund operators.
Nazarov et al. \cite{nptv17} established the boundedness
of Calder\'on--Zygmund operators on $L^p_W$
with norm bound $C[W]^{\frac 32}_{A_2}$ and, quite surprisingly,
Domelevo et al. \cite{dptv24} later proved that this exponent $\frac32$ is optimal.
Bownik and Cruz-Uribe \cite{bc22} extended
the Jones factorization and the Rubio de Francia
extrapolation theorems to the matrix weight setting.
For more studies on the boundedness of classical operators on the
matrix-weighted Lebesgue space $L^p_W$, we refer to
\cite{dhl20,dly21,llor23, llor24, n12, n25-2, nr18, v24}.
Meanwhile, matrix-weighted function spaces beyond Lebesgue
spaces have also attracted growing attention.
Frazier and Roudenko \cite{fr04, fr21} and Roudenko \cite{rou03, rou04}
systematically studied the  matrix-weighted Besov and Triebel--Lizorkin spaces.
Bu et al. \cite{bhyy1,bhyy2,bhyy3, byymz, byyzhang} and Yang et al.
\cite{YYZhang} considered and developed
more general matrix-weighted Besov-type and Triebel--Lizorkin-type spaces.
For more studies on matrix-weighted function spaces, we also refer
to \cite{bgx25,bx24-1, bx24-2, bx24-3, lyy24-1, lyy24-2,
n25, wgx25, wgx25b, wyy23}.
We also refer to survey articles \cite{byyz,c25}
for more recent results related to matrix weights.

Surprisingly, a systematic real-variable theory
for matrix-weighted Hardy spaces has long remained
unexplored. To solve this question, recently Bu et al.
\cite{bcyy} introduced and studied the
matrix-weighted Hardy space $H^p_W$ with matrix $A_p$-weight $W$ and
characterized them, respectively,
in terms of various maximal functions and atoms.
However, a dual theory of these Hardy spaces is still missing.
Recall that, in the scalar case, the dual space of the Hardy space $H^p$ is the Campanato space (see \cite{GR}); particularly,
the dual space of the Hardy space $H^1$ is the space  BMO (see \cite{FS}), which is a milestone
achievement in harmonic analysis.
The classical Campanato spaces were originally introduced
by Campanato in \cite{C64} and have  played an
important role in both harmonic analysis
(see, for example, \cite{ho19,jtyyz22iii,JN61,N10,N17,ns2012,TYY19,
YN21,YNS25,ZCTY23})
and partial differential equations
(see, for example, \cite{C66,CW21,CW23,DL19,NY19}).
In the scalar case, the weighted Campanato spaces have been studied by Garcia-Cuerva
\cite{G79} and Hu and Zhou \cite{HZ19},
and some recent variants of classical weighted Campanato spaces
and their applications are given by D. Yang and S. Yang \cite{YY10,YY11} and
Hu et al. \cite{HBLWY}.
Based on these, a natural question is what is the dual space of $H^p_W$.

Based on the known atomic and finite atomic characterizations of
$H^p_W$ obtained in \cite{bcyy}, in this article
we introduce a matrix-weighted Campanato space
that serves as the dual space of $H^p_W$.
More precisely, let $p\in(0,\infty)$, $q\in[1,\infty)$,
$s\in\mathbb Z_+$,
and $W$ be an $A_p$-matrix weight,
which in scalar case is exactly a
Muckenhoupt $A_{\max\{1,p\}}$ weight. By using the
reducing operators of $W$, we first introduce matrix-weighted
Campanato spaces $\mathcal L_{p,q,s,W}$.
Then, when $p\in(0,1]$, using the atomic and the
finite atomic characterizations of $H^p_W$, we prove that the dual space of
$H^p_W$ is precisely $\mathcal L_{p,q,s,W}$.
Applying this duality, we establish several equivalent characterizations
of $\mathcal L_{p,q,s,W}$. In addition, when $p\in(0,\infty)$, we
obtain a necessary and sufficient condition for the boundedness
of modified  Calder\'on--Zygmund operators on $\mathcal L_{p,q,s,W}$.
Combining this with the duality further yields a necessary
and sufficient condition for the boundedness of Calder\'on--Zygmund
operators on $H^p_W$ with $p\in(0,1]$.
This result complements the work in \cite{bcyy}, where only the
sufficiency of the boundedness of
Calder\'on--Zygmund operators on $H^p_W$ was given.

We point out that it is nontrivial to find a suitable definition of the matrix-weighted
Campanato space that correctly captures the duality of $H^p_W$.
Indeed, we need to  combine and make full use of the
reducing operators related to matrix weight $W$ under consideration
and the atomic and the finite atomic
decompositions of $H^p_W$.

The organization of the remainder of this article is as follows.

In Section \ref{ad},
let $p\in(0,\infty)$, $q\in[1,\infty)$,
$s\in\mathbb Z_+$,
and $W$ be an $A_p$-matrix weight.
We introduce the $\mathbb A$-matrix-weighted
Campanato spaces $\mathcal L_{p,q,s,\mathbb A}$ (see Definition \ref{mwcs})
and then show that, if $\mathbb A:=\{A_Q\}_{Q\in\mathscr{Q}}$
is a family of reducing operators of order $p$ for $W$, then
$\mathcal L_{p,q,s,\mathbb A}$ is independent of the
choice of $\mathbb A$, which we denote as $\mathcal L_{p,q,s,W}$
[see Remark \ref{re74}(i)].
When $p\in(0,1]$, using the atomic and the
finite atomic characterizations of the matrix-weighted Hardy
space $H^p_W$ (see Lemmas \ref{W-F-atom} and \ref{finatom}),
we establish the duality between $H^p_W$ and $\mathcal L_{p,q,s,W}$
(see Theorem \ref{dual}).
As a consequence of this duality, we obtain several equivalent
characterizations of $\mathcal L_{p,q,s,W}$
(see Corollary \ref{2247} and Theorem \ref{JN}).

In Section \ref{cz-l}, we recall the concept of the modified Calder\'on--Zygmund
operator related to a given Calder\'on--Zygmund
operator $T$ (see Definition \ref{mcz}) and
prove that this modified Calder\'on--Zygmund operator
is bounded on $\mathcal L_{p,q,s,W}$ if and only if
$T$  satisfies the vanishing conditions up to order $s$
(see Theorem \ref{btbbmc}).
Combining this and the duality with $p\in(0,1]$, we further show that
the Calder\'on--Zygmund operator $T$ is bounded on $H^p_W$
if and only if the adjoint operator $T^\ast$ of the Calder\'on--Zygmund
operator $T$ satisfies the vanishing conditions up to order $s$
(see Theorem \ref{CZ}).
This result complements the work in \cite{bcyy}, where only the
sufficiency for the boundedness of $T$ was given.

At the end of this section, we make some conventions on symbols.
Throughout this article, we work in $\mathbb{R}^n$,
which serves as our default underlying space
unless otherwise specified.
For any $p\in(0,\infty]$, the Lebesgue space $L^p$
has the usual meaning
and we define
$\|\vec f\|_{L^p}:=\|\,|\vec f|\,\|_{L^p}$
for all measurable
vector-valued functions $\vec f: \mathbb{R}^n\to\mathbb{C}^m$.
A cube  $Q\subset\mathbb{R}^n$ always has edges
parallel to the axes. For any cube $Q\subset\mathbb{R}^n$,
let $c_Q$ be its \emph{center} and $l(Q)$
denote its \emph{edge length}.
For any $r\in(0,\infty)$,
let $rQ$ denote the cube with the same
center as $Q$ and the edge length
$rl(Q)$.
We use $\mathbf{0}$ to
denote the \emph{origin} of ${\mathbb{R}^n}$.
For any $p\in(0,\infty)$, let $L^p_{\mathrm{loc}}$ denote the
set of all $p$-locally integrable functions on $\mathbb R^n$.
For any measurable set $E$ in
$\mathbb{R}^n$ with $|E|\in(0,\infty)$ and for any
$f\in L^1_{\mathrm{loc}}$
(the set of all locally integrable functions on $\mathbb{R}^n$), let
$$
\fint_E f(x)\,dx:=\frac{1}{|E|}\int_{E}f(x)\,dx
$$
and we denote by $\mathbf{1}_E$ the
\emph{characteristic function} of $E$ and by $E^\complement$
its \emph{complementary set}.
For any $x\in\mathbb{R}^n$ and $r\in(0,\infty)$, let
$$B(x,r):=\left\{y\in\mathbb{R}^n:  |x-y|<r\right\}$$
be the ball with center $x$ and radius $r$.
Let $\mathbb N:=\{1,2,\ldots\}$ and $\mathbb Z_+:=\mathbb N\cup\{0\}$.
For any
$\alpha=(\alpha_1,\ldots,\alpha_n)\in\mathbb Z_+^n$
and $x=(x_1,\ldots,x_n)$,
let $|\alpha|:=\alpha_1+\cdots+\alpha_n$
and $\partial^\alpha:=(\frac{\partial}{\partial x_1})^{\alpha_1}\cdots
(\frac{\partial}{\partial x_n})^{\alpha_n}$.
For any $p\in(0, \infty]$, let
$p':=\frac{p}{p-1}$ if $p\in(1, \infty]$
and $p':=\infty$ if $p\in(0, 1]$ be the \emph{conjugate index} of $p$.
We always denote by $C$ a \emph{positive constant}
which is independent of the main parameters involved,
but it may vary from line to line.
The symbol $A\lesssim B$ means that $A\le CB$ for
some positive constant $C$, while $A\sim B$ means $A\lesssim B\lesssim A$.
Also, for any $\alpha\in\mathbb{R}$, we use $\lfloor\alpha\rfloor$
(resp. $\lceil\alpha\rceil$) to denote the largest (resp. smallest) integer
not greater (resp. less) than $\alpha$.
Finally, in all proofs
we consistently retain the symbols
introduced in the original theorem (or related statement).

\section{Matrix-Weighted
Campanato Spaces\label{ad}}

For any $m,n\in\mathbb{N}$, let $M_{m,n}(\mathbb{C})$ denote
the set of all $m\times n$ complex-valued matrices
and we simply write $M_{m,m}(\mathbb{C})$ as $M_{m}(\mathbb{C})$.
The zero matrix in $M_{m,n}(\mathbb{C})$ is denoted by $O_{m,n}$
and $O_{m,m}$ is simply written as $O_m$.
For any $A\in M_m(\mathbb{C})$, let
$$\|A\|:=\sup_{\vec z\in\mathbb{C}^m,\,|\vec z|=1}|A\vec z|.$$
Let $A^*$ denote the \emph{conjugate transpose} of $A$.
A matrix $A\in M_m(\mathbb{C})$ is said to be \emph{positive definite}
if, for any $\vec z\in\mathbb{C}^m\setminus\{\vec{0}\}$, $\vec z^*A\vec z>0$,
and $A$ is said to be \emph{nonnegative definite} if,
for any $\vec z\in\mathbb{C}^m$, $\vec z^*A\vec z\geq0$.
Here, and thereafter, $\vec 0$ denotes the zero vector of $\mathbb C^m$.
The matrix $A\in M_m(\mathbb C)$ is called a \emph{unitary matrix} if $A^*A=I_m$,
where $I_m$ is the \emph{identity matrix}.

In the remainder of this article, we \emph{always fix} $m\in\mathbb{N}$.
Let $A\in M_m(\mathbb{C})$ be a positive definite matrix
and have eigenvalues $\{\lambda_i\}_{i=1}^m$.
Using \cite[Theorem 2.5.6(c)]{hj13},
we find that there exists a unitary matrix $U\in M_m(\mathbb{C})$ such that
\begin{equation}\label{500}
A=U\operatorname{diag}\,(\lambda_1,\ldots,\lambda_m)U^*.
\end{equation}
The following definition
can be found in \cite[(6.2.1)]{hj94}
(see also \cite[Definition 1.2]{h08}).

\begin{definition}
Let $A\in M_m(\mathbb{C})$ be a positive definite matrix
and have eigenvalues $\{\lambda_i\}_{i=1}^m$.
For any $\alpha\in\mathbb{R}$, define
$$A^\alpha:=U\operatorname{diag}(\lambda_1^\alpha,\ldots,
\lambda_m^\alpha)U^*,$$
where $U$ is the same as in \eqref{500}.
\end{definition}

\begin{remark}
From \cite[p.\,408]{hj94}, we infer that $A^\alpha$
is independent of the choices of the order of $\{\lambda_i\}_{i=1}^m$ and $U$,
and hence $A^\alpha$ is well-defined.
\end{remark}

Now, we recall the concept of matrix weights
(see, for example, \cite{nt96,tv97,v97}).

\begin{definition}\label{MatrixWeight}
A matrix-valued function $W:  \mathbb{R}^n\to M_m(\mathbb{C})$ is called
a \emph{matrix weight} if $W(x)$ is nonnegative definite for any $x\in\mathbb R^n$,
$W(x)$ is invertible for almost every $x\in{\mathbb{R}^n}$, and the
entries of $W$ are all locally integrable.
\end{definition}

Corresponding to the scalar Muckenhoupt $A_p({\mathbb{R}^n})$ class,
Nazarov and Treil \cite{nt96} and Volberg \cite{v97} originally
independently introduced $A_p({\mathbb{R}^n},\mathbb{C}^m)$-matrix
weights with $p\in (1,\infty)$. The following version of
$A_p({\mathbb{R}^n},\mathbb{C}^m)$-matrix weights was originally introduced
by Frazier and Roudenko \cite{fr04} with $p\in (0,1]$ and
Roudenko \cite{rou03} with $p\in (1,\infty)$ (see also \cite[p.\,490]{fr21}).

\begin{definition}\label{def ap}
Let $p\in(0,\infty)$. A matrix weight $W$ on $\mathbb{R}^n$
is called an $A_p({\mathbb{R}^n},\mathbb{C}^m)$-\emph{matrix weight}
if $W$ satisfies that, when $p\in(0,1]$,
$$
[W]_{A_p({\mathbb{R}^n},\mathbb{C}^m)}
:=\sup_{\mathrm{cube}\,Q}\mathop{\mathrm{\,ess\,sup\,}}_{y\in Q}
\fint_Q\left\|W^{\frac{1}{p}}(x)W^{-\frac{1}{p}}(y)\right\|^p\,dx
<\infty
$$
or that, when $p\in(1,\infty)$,
$$
[W]_{A_p({\mathbb{R}^n},\mathbb{C}^m)}
:=\sup_{\mathrm{cube}\,Q}
\fint_Q\left[\fint_Q\left\|W^{\frac{1}{p}}(x)W^{-\frac{1}{p}}(y)\right\|^{p'}
\,dy\right]^{\frac{p}{p'}}\,dx
<\infty.
$$
\end{definition}

When $m=1$, $A_{p}({\mathbb{R}^n},\mathbb{C}^m)$-matrix weights
in this case reduce to classical scalar Muckenhoupt
$A_{\max\{1,p\}}({\mathbb{R}^n})$ weights.
In what follows, if there exists no confusion, we denote
$A_p({\mathbb{R}^n},\mathbb{C}^m)$ simply by $A_p$.

To introduce matrix-weighted Campanato spaces,
we also need the concept of reducing operators,
one of the most important tools in
the study of matrix weights, which was originally introduced by
Volberg in \cite[(3.1)]{v97}.

\begin{definition}\label{reduce}
Let $p\in(0,\infty)$, $W$ be a matrix weight,
and $E\subset\mathbb{R}^n$ a bounded measurable set satisfying $|E|\in(0,\infty)$.
The matrix $A_E\in M_m(\mathbb{C})$ is called a
\emph{reducing operator} of order $p$ for $W$
if $A_E$ is positive definite and,
for any $\vec z\in\mathbb{C}^m$,
\begin{equation}\label{equ_reduce}
\left|A_E\vec z\right|
\sim\left[\fint_E\left|W^{\frac{1}{p}}(x)\vec z\right|^p\,dx\right]^{\frac{1}{p}},
\end{equation}
where the positive equivalence constants depend only on $m$ and $p$.
\end{definition}

Next, we introduce $\mathbb A$-matrix-weighted Campanato spaces
and matrix-weighted Campanato spaces.
For any $s\in\mathbb Z_+$,
let $\mathcal{P}_s$ denote the set of all polynomials on ${\mathbb{R}^n}$
of total degree not greater than $s$.

\begin{definition}\label{mwcs}
Let $p\in(0,\infty)$, $q\in[1,\infty)$, $s\in\mathbb Z_+$, and
$\mathbb A:=\{A_Q\}_{\text{cube }Q}$ be a family of positive definite matrices.
The $\mathbb A$-\emph{matrix-weighted Campanato space}
$\mathcal L_{p,q,s,\mathbb A}$
is defined to be the set of all
$\vec g\in (L^q_{\rm{loc}})^m$ such that
\begin{align*}
\left\|\vec g\right\|_{\mathcal L_{p,q,s,\mathbb A}}:=
\sup_{\mathrm{cube}\,Q}\inf_{\vec P\in(\mathcal{P}_s)^m}|Q|^{\frac1{q'}-\frac1p}
\left\{\int_{Q}\left|
A_Q^{-1}\left[\vec g(x)-\vec P(x)\right]
\right|^q\,dx\right\}^{\frac1q}<\infty.
\end{align*}
\end{definition}

\begin{remark}\label{re74}
Let $p\in(0,\infty)$, $q\in[1,\infty)$, $s\in\mathbb Z_+$, and $W\in A_p$.
\begin{enumerate}[\rm(i)]
\item
Assume that both $\mathbb A:=\{A_Q\}_{\rm{cube}\,Q}$
and $\mathbb A':=\{A'_Q\}_{\rm{cube}\,Q}$
are two families of reducing operators of order $p$ for $W$.
Then, using \cite[Propositions 2.14]{byyz},
we conclude that, for any cube $Q$,
\begin{align}\label{p0infty}
\left\|A_Q (A'_Q)^{-1}\right\|&\sim\left[\fint_Q\left\|
W^{\frac{1}{p}}(x)(A'_Q)^{-1}\right\|^p\,dx\right]^{\frac{1}{p}},
\end{align}
where the positive equivalence constants depend only on $m$ and $p$.
If $p\in(0,1]$, from \eqref{p0infty} and
\cite[Lemma 2.6(iii) and Proposition 2.4]{byyz},
we infer that
\begin{align}\label{ifp01}
\left\|A_Q (A'_Q)^{-1}\right\|\sim\left[\fint_Q\mathop{\mathrm{\,ess\,sup\,}}_{y\in Q}\left\|
W^{\frac{1}{p}}(x)W^{-\frac1p}(y)\right\|^p\,dx\right]^{\frac{1}{p}}
\lesssim[W]_{A_p}^{\frac1p},
\end{align}
where the implicit positive constants depend only on $m$ and $p$.
If $p\in(1,\infty)$, applying \eqref{p0infty} and
\cite[Lemma 2.6(ii)]{byyz},
we find that
\begin{align*}
\left\|A_Q (A'_Q)^{-1}\right\|\sim\left\{
\fint_Q\left[\fint_Q\left\|W^{\frac{1}{p}}(x)W^{-\frac{1}{p}}(y)\right\|^{p'}
\,dy\right]^{\frac{p}{p'}}\,dx\right\}^{\frac{1}{p}}
\lesssim[W]_{A_p}^{\frac1p},
\end{align*}
where the implicit positive constants depend only on $m$ and $p$.
This, together with \eqref{ifp01},
further implies that $\mathcal L_{p,q,s,\mathbb A}$ and
$\mathcal L_{p,q,s,\mathbb A'}$
coincide with equivalent norms. Based on this,
we define the \emph{matrix-weighted Campanato space} 
\begin{equation}\label{c1}
\mathcal L_{p,q,s,W}:=\mathcal L_{p,q,s,\mathbb A}.
\end{equation}

\item
By the definition of $\mathcal L_{p,q,s,W}$, we easily find that
$\|\vec f\|_{\mathcal L_{p,q,s,W}}=0$ if and only if $\vec f\in(\mathcal{P}_s)^m$.
\item
Let $m=1$.
For any cube $Q$, we write 
$$W(Q):=\int_QW(x)\,dx.$$
Observe that, in this case, by Definition
\ref{reduce}, one obviously has, for any cube
$Q$, 
\begin{equation}\label{a1}
A_Q\sim \left[\frac {W(Q)}{|Q|}\right]^{\frac 1p},
\end{equation}
where the positive equivalence constants
depend only on $p$; hence, for any
$g\in L^q_{\rm{loc}}$,
\begin{align*}
\left\|g\right\|_{\mathcal L_{p,q,s,W}}=
\sup_{\mathrm{cube}\,Q}\inf_{P\in \mathcal{P}_s}
\frac{|Q|}{W(Q)^\frac1p}\left\{\fint_{Q}\left|g(x)-P(x)
\right|^q\,dx\right\}^{\frac1q}.
\end{align*}
Thus, in this case, $\mathcal L_{p,q,s,W}$
reduces to the scalar weighted Campanato space,
which is a special case of \cite[Definition 1.11]{yyy} with
$X:=L^p_W$.
\end{enumerate}
\end{remark}

Now, we show an equivalent characterization of
$\mathcal L_{p,q,s,\mathbb A}$.
To this end, we introduce the following symbol.
Here, and thereafter, for any $s\in\mathbb Z_+$
and any compact set $E\subset\mathbb R^n$, let
$$\Pi_E^s:L^1(E)\rightarrow\mathcal{P}_s$$
be the \emph{natural projection} satisfying, for any
$f\in L^1(E)$ and $q\in\mathcal{P}_s$,
\begin{align}\label{fy}
\int_E\Pi_E^sf(x)q(x)\,dx=
\int_Ef(x)q(x)\,dx.
\end{align}
In what follows, for any $\vec f:=(f_1,\ldots,f_m)$,
we let 
$$\Pi^s_E\vec f:=\left(\Pi^s_Ef_1,\ldots,\Pi^s_Ef_m\right).$$

\begin{lemma}\label{lPi}
Let all the symbols be the same as in Definition \ref{mwcs}.
Then, for any $\vec g\in \mathcal L_{p,q,s,\mathbb A}$,
\begin{align*}
\left\|\vec g\right\|_{\mathcal L_{p,q,s,\mathbb A}}
\sim \sup_{\mathrm{cube}\,Q}|Q|^{\frac1{q'}-\frac1p}
\left\{\int_{Q}\left|A_Q^{-1}\left[\vec g(x)-\Pi^s_Q\vec g(x)\right]
\right|^q\,dx\right\}^{\frac1q},
\end{align*}
where the positive equivalence constants are independent of $\vec f$.
\end{lemma}

\begin{proof}
From \cite[(2.12)]{jtyyz22iii},
we deduce that
\begin{align*}
\left\|\vec g\right\|_{\mathcal L_{p,q,s,\mathbb A}}&=
\sup_{\mathrm{cube}\,Q}\inf_{A_Q\vec P\in(\mathcal{P}_s)^m}|Q|^{\frac1{q'}-\frac1p}
\left\{\int_{Q}\left|
A_Q^{-1}\vec g(x)-\vec P(x)
\right|^q\,dx\right\}^{\frac1q}\\
&\sim\sup_{\mathrm{cube}\,Q}|Q|^{\frac1{q'}-\frac1p}
\left\{\int_{Q}\left|
A_Q^{-1}\vec g(x)-\Pi^s_Q\left(A_Q^{-1}\vec g\right)(x)
\right|^q\,dx\right\}^{\frac1q}.
\end{align*}
This finishes the proof of Lemma \ref{lPi}.
\end{proof}

\begin{remark}
Let $m=1$ and $W\in A_1$. Using \eqref{a1} with $p=1$ 
(namely $A_Q\sim\frac{W(Q)}{|Q|}$) and using the definition
\eqref{c1} of $\mathcal L_{p,q,s,W}$ in Remark \ref{re74}(i),
we find that, for any
$g\in\mathcal L_{1,1,0,W}$,
\begin{align*}
\left\|g\right\|_{\mathcal L_{1,1,0,W}}\sim
\sup_{\mathrm{cube}\,Q}\inf_{P\in \mathcal{P}_s}
\frac{1}{W(Q)}\int_{Q}\left|g(x)-P(x)
\right|\,dx
\end{align*}
with the positive equivalence constants independent of $g$,
which, by Lemma \ref{lPi}, reduces to the
definition in \cite[(1.1)]{MW75}.
\end{remark}

In the remainder of this section, we show the
duality between matrix-weighted Hardy spaces and
matrix-weighted Campanato spaces.
To this end, we first give several concepts.
Let $\mathcal{S}$ be the set of
all Schwartz functions on $\mathbb R^n$ and
$\mathcal{S}'$ be the set of
all tempered distributions on $\mathbb R^n$.
For any $N\in\mathbb Z_+$, let
$$\mathcal{S}_N:=\left\{\phi\in\mathcal{S}:
\|\phi\|_{\mathcal{S}_N}\le1\right\},$$
where
$$\left\|\phi\right\|_{\mathcal{S}_N}:=
\sup_{\alpha\in\mathbb Z_+^n,\,|\alpha|\le N+1}
\sup_{x\in\mathbb R^n}(1+|x|)^{N+n+1}
\left|\partial^\alpha\phi(x)\right|.$$
For any function $\phi$ on $\mathbb R^n$ and $t\in(0,\infty)$, define
$\phi_t(\cdot):=\frac{1}{t^n}\phi(\frac{\cdot}{t})$.
Let $a\in(0,\infty)$ and $N\in\mathbb Z_+$.
Then, for any $\vec f\in(\mathcal S')^m$,
the \emph{matrix-weighted grand non-tangential maximal function}
$(M_{a,N}^*)^p_{W}(\vec f)$ of $\vec f$ is defined by setting, for any $x\in\mathbb R^n$,
\begin{align*}
\left(M_{a,N}^*\right)^p_{W}\left(\vec f\right)(x):=
\sup_{\phi\in\mathcal{S}_N}\sup_{t\in(0,\infty)}\sup_{y\in B(x,at)}
\left|W^{\frac{1}{p}}(x)\phi_t*\vec f(y)\right|.
\end{align*}

We now recall the following definition of
matrix-weighted Hardy spaces, which is exactly \cite[Definition 2.4]{bcyy}.

\begin{definition}\label{de-HW}
Let $p\in(0,\infty)$, $N\in\mathbb Z_+$, and $W$ be a matrix weight.
The \emph{matrix-weighted Hardy space} $H_{W,N}^p$
is defined to be the set of all $\vec f\in(\mathcal{S}')^m$ such that
$(M_{1,N}^*)^p_{W}(\vec f)\in L^p$
equipped with the quasi-norm
$\|\vec f\|_{H_{W,N}^p}
:=\|(M_{1,N}^*)^p_{W}(\vec f)\|_{L^p}.$
\end{definition}

Based on \cite[Theorem 2.10]{bcyy}, in what follows we simply write
$H_{W,N(W)}^p$ as $H_W^p$,
where $W\in A_p$ and $N(W):=\lfloor\frac{2n}p\rfloor+1$.
The following concept of atoms is precisely \cite[Definition 3.1]{bcyy}.
\begin{definition}\label{F-atom}
Let $p\in(0,\infty)$, $q\in[1,\infty]$, $s\in{\mathbb Z}_+$,
and $W$ be a matrix weight.
A function $\vec a$ is called a \emph{$(p,q,s)_{W}$-atom}
supported in a cube $Q$ if
\begin{enumerate}
\item[\rm (i)] $\operatorname{supp}\vec a\subset Q$,
\item[\rm (ii)]
$$\left\{\int_Q\left[\int_{Q}|W^{\frac1p}(y)\vec a(x)
|^p\,dy\right]^{\frac qp}\,dx\right\}^{\frac1q}\leq|Q|^{\frac1q},$$
\item[\rm (iii)]
for any $\gamma:=(\gamma_1,\ldots,\gamma_n)\in\mathbb{Z}_+^n$
with $|\gamma|:=\gamma_1+\cdots+\gamma_n\le s$,
$$\int_{\mathbb{R}^n}x^\gamma \vec a(x)\,dx=\vec{0},$$
where $x^\gamma:=x_1^{\gamma_1}\cdots x_n^{\gamma_n}$
for any $x:=(x_1,\ldots,x_n)\in\mathbb{R}^n$.
\end{enumerate}
\end{definition}

\begin{remark}\label{my}
Using \eqref{equ_reduce}, we find that Definition \ref{F-atom}(ii) is equal to
$$\left[\int_Q\left|A_Q\vec a(x)\right|^q\,dx\right]^\frac1q\lesssim|Q|^{\frac1q-\frac1p}$$
with the implicit positive constant independent of $\vec a$ and $Q$.
\end{remark}

Next, we present the following definition of
the finite atomic matrix-weighted Hardy space,
which is exactly \cite[Definition 4.1]{bcyy}.

\begin{definition}\label{de-fin}
Let $p\in(0,1]$, $q\in[1,\infty]$, $s\in\mathbb Z_+$, and $W$ be a matrix weight.
The \emph{finite atomic matrix-weighted Hardy space}
$H^{p,q,s}_{W,\rm{fin}}$ is defined to be
the set of all finite linear combinations of $(p,q,s)_W$-atoms
equipped with the quasi-norm
\begin{align*}
\left\|\vec f\right\|_{H^{p,q,s}_{W,\rm{fin}}}:=\inf\left\{
\left(\sum_{k=1}^{N}|\lambda_k|^p\right)^{\frac1p}:
N\in\mathbb N,\,\vec f=\sum_{k=1}^{N}\lambda_k\vec a_k,\,
\left\{\lambda_k\right\}_{k=1}^N\subset[0,\infty)
\right\},
\end{align*}
where the infimum is taken over all finite linear combinations of
$\vec f$ in terms of $(p,q,s)_W$-atoms.
\end{definition}

The main result of this section is the following
dual theorem of $H^p_W$.
In what follows, let $\mathcal C$ denote the set of all continuous
functions on $\mathbb R^n$.

\begin{theorem}\label{dual}
Let $p\in(0,1]$, $q\in[1,\infty)$,
$s\in[\lfloor n(\frac{1}{p}-1)\rfloor,\infty)\cap\mathbb Z_+$,
and $W\in A_p$.
Then the dual space of $H^p_W$, denoted by $(H^p_W)^*$, is $\mathcal L_{p,q,s,W}$
in the following sense:
\begin{enumerate}[\rm(i)]
\item
Suppose $\vec g\in\mathcal L_{p,q,s,W}$. Then the linear functional
\begin{align}\label{fl}
L_g: vec f\longmapsto L_g\left(\vec f\right):=\int_{\mathbb R^n}\vec f(x)\vec g(x)\,dx,
\end{align}
initially defined for any $\vec f\in H^{p,q',s}_{W,\rm{fin}}\cap\mathcal C$,
has a bounded extension to $H^p_W$.

\item
Conversely, any continuous linear functional on $H^p_W$
arises as in \eqref{fl} with a unique $\vec g\in\mathcal L_{p,q,s,W}$.
\end{enumerate}
\end{theorem}

\begin{remark}
Let $m=1$ and $W\in A_1$. Using Remark \ref{re74}(iii), we find that,
$\mathcal L_{p,q,s,W}$ reduces to the scalar weighted Campanato space.
Theorem \ref{dual} in this case coincides with
\cite[Theorem 1.12]{yyy} with $X:=L^p_W$.
If we further assume that $W\equiv1$,
then $\mathcal L_{p,q,s,W}$ reduces to the scalar
unweighted Campanato space and Theorem \ref{dual} in this
case coincides with \cite[Theorem (2.7)]{TW80}.
\end{remark}

To prove this theorem, we need some technical lemmas.
The following lemma is the reverse H\"older inequality
associated with matrix weights, which is a direct consequence of
\cite[Proposition 2.13(ii)]{bhyyNew} by using the fact that
any matrix $A_p$ weight is also a matrix $A_{p,\infty}$ weight
(see \cite[Definition 2.7]{bhyyNew} for the definition of matrix 
$A_{p,\infty}$ weights).

\begin{lemma}\label{86}
Let $p\in(0,\infty)$, $W\in A_{p}$,
and $\mathbb A:=\{A_Q\}_{Q\in\mathscr{Q}}$ be a family of
reducing operators of order $p$ for $W$.
Then there exist a positive constant $r\in(1,\infty)$, depending only on $n$, $m$, $p$,
and $[W]_{A_{p}}$, and a positive constant $C$,
depending only on $m$ and $p$, such that
\begin{equation}\label{kappa}
\sup_{\mathrm{cube}\,Q}
\left[\fint_Q\left\|W^{\frac{1}{p}}(x)A_Q^{-1}\right\|^{pr}\,dx\right]^{\frac{1}{pr}}
\le C.
\end{equation}
\end{lemma}

Let
\begin{align}\label{r_w}
r_W:=\sup\left\{r:  r\text{ satisfies \eqref{kappa}}\right\}.
\end{align}

The following lemma can also be directly deduced from \cite[Theorem 3.5]{bcyy}
by using the facts that any matrix $A_p$ weight is also a matrix $A_{p,\infty}$ weight
and the upper dimension of matrix $A_p$ weights when $p\in(0,1]$
equals $0$ (see \cite[Definition 2.8]{bcyy} for the definition of the 
upper dimension of matrix $A_{p,\infty}$ weights).

\begin{lemma}\label{W-F-atom}
Let $p\in(0,1]$, $W\in A_p$, and
$s\in[\lfloor n(\frac{1}{p}-1)\rfloor,\infty)\cap\mathbb Z_+$.
Then the following two statements hold.
\begin{enumerate}[{\rm(i)}]
\item Let $q\in(\max\{1,\frac{r_W}{r_W-1}p\},\infty]$
with $r_W$ the same as in \eqref{r_w}.
For any $\{\lambda_k\}_{k\in{\mathbb Z}}\in l^p$ and
any $(p,q,s)_W$-atoms $\{\vec a_k\}_{k\in\mathbb Z}$,
there exists $\vec f\in H^p_W$ such that
\begin{equation}\label{e2.19}
\vec f=\sum_{k\in\mathbb Z}\lambda_k\vec a_k
\end{equation}
in both $H^p_W$ and $(\mathcal{S}')^m$.
Moreover, there exists a positive constant $C$, independent of both
$\{\lambda_k\}_{k\in\mathbb Z}$ and $\{\vec a_k\}_{k\in\mathbb Z}$, such that
$$\left\|\vec f\right\|_{H^p_W}\le
C\|\{\lambda_k\}_{k\in\mathbb Z}\|_{l^p}.$$

\item
For any $\vec f\in H^p_W$, there exist
a sequence $\{\lambda_k\}_{k\in\mathbb Z}\in l^p$
and a sequence of $(p,\infty,s)_W$-atoms $\{\vec a_k\}_{k\in\mathbb Z}$ such that
\eqref{e2.19} holds in both $H^p_W$ and $(\mathcal{S}')^m$.
Moreover, there exists a positive constant $C$, independent of $\vec f$,
such that
$$\|\{\lambda_k\}_{k\in\mathbb Z}\|_{l^p}
\le C\left\|\vec f\right\|_{H^p_W}.$$
\end{enumerate}
\end{lemma}

The following lemma is precisely \cite[Theorem 4.2]{bcyy}.

\begin{lemma}\label{finatom}
Let $p\in(0,1]$, $W\in A_p$,
and $s\in[\lfloor n(\frac{1}{p}-1)\rfloor,\infty)\cap\mathbb Z_+$.
Then $\|\cdot\|_{H^{p,\infty,s}_{W,\rm{fin}}}$ and
$\|\cdot\|_{H^p_W}$ are equivalent quasi-norms
on the space $H^{p,\infty,s}_{W,\rm{fin}}\cap\mathcal C$.
\end{lemma}

The following lemma is exactly \cite[Lemma 2.29]{bhyy1}.

\begin{lemma}\label{p01}
Let $p\in(0,\infty)$, let $W\in A_p$ have the $A_p$-dimension $d\in[0,n)$,
and let $\{A_Q\}_{\mathrm{cube}\,Q}$ be a family of
reducing operators of order $p$ for $W$. If $p\in(1,\infty)$,
let further $\widetilde W:=W^{-\frac 1{p-1}}$
(which belongs to $A_{p'}$) have the $A_{p'}$-dimension $\widetilde d\in[0,n)$,
while, if $p\in(0,1]$, let $\widetilde d:=0$.
Here the $A_p$-dimension and the $A_{p'}$-dimension are the same as in
\cite[Definition 2.23]{bhyy1}.
Then there exists a positive constant $C$ such that,
for any cubes $Q$ and $R$ in $\mathbb{R}^n$,
\begin{equation}\label{e2.9}
\left\|A_QA_R^{-1}\right\|
\leq C\max\left\{\left[\frac{\ell(R)}{\ell(Q)}\right]^{\frac{d}{p}},
\left[\frac{\ell(Q)}{\ell(R)}\right]^{\widetilde d(1-\frac{1}{p})}\right\}
\left[1+\frac{|c_Q-c_R|}{\ell(Q)\vee\ell(R)}\right]^{\frac{d}{p}+\widetilde d(1-\frac{1}{p})}.
\end{equation}
\end{lemma}

\begin{remark}
In \cite[Lemma 2.47]{bhyy1}, Bu et al. proved that the indices appearing in the right-hand side
of inequality \eqref{e2.9} is sharp.
\end{remark}

\begin{lemma}\label{hd}
Let $p\in(0,1]$, $q\in(\max\{1,\frac{r_W}{r_W-1}p\},\infty]$,
$s\in[\lfloor n(\frac{1}{p}-1)\rfloor,\infty)\cap\mathbb Z_+$,
and $W\in A_p$. Then, for any continuous linear functional
$L$ on $H^p_W$,
\begin{align*}
\|L\|_{(H^p_W)^*}:=\sup\left\{\left|
L\left(\vec f\right)
\right|: \left\|\vec f\right\|_{H^p_W}\leq1\right\}
\sim\sup\left\{\left|L\left(\vec a\right)
\right|: \vec a\text{ is a }(p,q,s)_W\text{-atom}\right\},
\end{align*}
where the positive equivalence constants are independent of $L$.
\end{lemma}

\begin{proof}
By Lemma \ref{W-F-atom}(i), we conclude that,
for any $(p,q,s)_W$-atom $\vec a$, $\|\vec a\|_{H^p_W}\lesssim1$ and
hence
\begin{align*}
\sup\left\{\left|
L\left(\vec f\right)
\right|:  \left\|\vec f\right\|_{H^p_W}\leq1\right\}
\gtrsim\sup\left\{\left|L\left(\vec a\right)
\right|:  \vec a\text{ is a }(p,q,s)_W\text{-atom}\right\}.
\end{align*}
Conversely, let $\vec f\in H^p_W$ satisfy $\|\vec f\|_{H^p_W}\leq 1$.
From Lemma \ref{W-F-atom}(ii), we infer that
there exist
a sequence $\{\lambda_k\}_{k\in\mathbb Z}\in l^p$
and a sequence of $(p,\infty,s)_W$-atoms $\{\vec a_k\}_{k\in\mathbb Z}$ such that
$\vec f=\sum_{k\in\mathbb Z}\lambda_k\vec a_k$ in both $H^p_W$ and $(\mathcal{S}')^m$
and $\|\{\lambda_k\}_{k\in\mathbb Z}\|_{l^p}\lesssim1$.
Combining this and the continuity of $L$, we find that
\begin{align*}
\left|L\left(\vec f\right)\right|&\leq\left(\sum_{k\in\mathbb Z}
|\lambda_k|\right)\left|L\left(\vec a_k\right)\right|
\leq\left\|\{\lambda_k\}_{k\in\mathbb Z}\right\|_{l^p}
\sup\left\{\left|L\left(\vec a\right)
\right|:  \vec a\text{ is a }(p,q,s)_W\text{-atom}\right\}\\
&\lesssim\sup\left\{\left|L\left(\vec a\right)
\right|:  \vec a\text{ is a }(p,q,s)_W\text{-atom}\right\},
\end{align*}
which completes the proof of the converse direction
and hence Lemma \ref{hd}.
\end{proof}

Now, we prove Theorem \ref{dual}. In what follows, we use $f|_E$
to denote a function $f$ restricted to the set $E$ in $\mathbb{R}^n$.

\begin{proof}[Proof of Theorem \ref{dual}]
We first prove (i).
Let $\vec g\in\mathcal L_{p,q,s,W}$ and $\vec a$ be a $(p,q',s)_W$-atom
supported in cube $Q$. By the H\"older inequality, Remark \ref{my}, and
Definition of $\mathcal L_{p,q,s,W}$, we find that
\begin{align}\label{yy}
\left|\int_{\mathbb R^n}\vec a(x)\vec g(x)\,dx\right|
&=\inf_{\vec P\in(\mathcal{P}_s)^m}
\left|\int_{\mathbb R^n}\vec a(x)\left[\vec g(x)-\vec P(x)\right]\,dx\right|\nonumber\\
&\leq\left\|A_Q\vec a\right\|_{L^{q'}}\inf_{\vec P\in(\mathcal{P}_s)^m}
\left\{\int_{Q}\left|A_Q^{-1}
\left[\vec g(x)-\vec P(x)\right]\right|^{q}\,dx\right\}^{\frac1{q}}\nonumber\\
&\leq|Q|^{\frac1{q'}-\frac1p}\inf_{\vec P\in(\mathcal{P}_s)^m}
\left\{\int_{Q}\left|A_Q^{-1}
\left[\vec g(x)-\vec P(x)\right]\right|^{q}\,dx\right\}^{\frac1{q}}\nonumber\\
&\leq\left\|\vec g\right\|_{\mathcal L_{p,q,s,W}}.
\end{align}
From Definition \ref{de-fin}, it follows that, for any
$\vec f\in H^{p,q',s}_{W,\rm{fin}}$, there exist $N\in\mathbb N$,
a sequence
$\{\lambda_k\}_{k=1}^N$ in $[0,\infty)$, and a sequence of $(p,q',s)_W$-atoms
$\{\vec a_k\}_{k=1}^N$ such that $\vec f=\sum_{k=1}^{N}\lambda_k\vec a_k$
and $\|\{\lambda_k\}_{k=1}^N\|_{l_p}\leq2\|\vec f\|_{H^{p,q',s}_{W,\rm{fin}}}$.
Combining this and \eqref{yy}, we find that
\begin{align}\label{spm}
L_{\vec g}\left(\vec f\right)&=\left|\int_{\mathbb R^n}\vec f(x)\vec g(x)\,dx\right|
\leq\sum_{k=1}^{N}|\lambda_k|\left|\int_{\mathbb R^n}
\vec a_k(x)\vec g(x)\,dx\right|\nonumber\\
&\leq\left\|\{\lambda_k\}_{k=1}^N\right\|_{l_p}
\left\|\vec g\right\|_{\mathcal L_{p,q,s,W}}
\leq2\left\|\vec f\right\|_{H^{p,q',s}_{W,\rm{fin}}}
\left\|\vec g\right\|_{\mathcal L_{p,q,s,W}}.
\end{align}
From the proof of \cite[Theorem 4.4]{bcyy}, we deduce that
$H^{p,q',s}_{W,\rm{fin}}\cap\mathcal C$ is dense in $H^p_W$, which,
together with \eqref{spm} and Lemma \ref{finatom}, further implies
that (i) holds.

Next, we prove (ii).
For any $r\in(1,\infty]$, any $s\in\mathbb Z_+$, any compact set $E\subset\mathbb R^n$,
and any matrix $A\in M_m(\mathbb C)$, let
$L^r(E,A)$ be the set of all measurable vector-valued functions $\vec f$
supported in $E$
such that $\|\vec f\|_{L^r(E,A)}:=\|A\vec f\|_{L^r}<\infty$
and, for any $s\in\mathbb Z_+$, let
\begin{align*}
L^r_{s}(E,A):=\left\{\vec f\in L^r(E,A):  \Pi^s_E\left(
\vec f\right)=\vec 0\right\}
\end{align*}
equipped with the norm
$\|\vec f\|_{L^r_s(E,A)}:=\|\vec f\|_{L^r(E,A)}$.
Then $L^r_{s}(E,A)$ is a closed subspace of $L^r(E)$.
Fix a cube $Q$. Using Remark \ref{my}, we conclude that,
for any $\vec f\in L^{q'}_{s}(Q,A_Q)$ with $\|\vec f\|_{L^{q'}_s(Q,A_Q)}\neq 0$,
\begin{align*}
\vec a:=|Q|^{\frac1{q'}-\frac1p}
\left\|\vec f\right\|_{L^{q'}(Q,A_Q)}^{-1}\vec f
\end{align*}
is a $(p,{q'},s)_W$-atom.
From this and Lemma \ref{hd}, we deduce that, for any $L\in(H^p_W)^*$
and any $\vec f\in L^{q'}_{s}(Q,A_Q)$,
\begin{align}\label{wt}
\left|L\left(\vec f\right)\right|\lesssim|Q|^{-\frac1{q'}+\frac1p}\|L\|_{(H^p_W)^*}
\left\|\vec f\right\|_{L^{q'}(Q,A_Q)},
\end{align}
which further implies that $L$ is a bounded linear functional on $L^{q'}_s(Q,A_Q)$.
This, together with the Hahn-Banach theorem, further implies that
$L$ can be extended to the space $L^{q'}(Q,A_Q)$
without increasing its norm.
Using the fact that,
for any $r\in [1,\infty]$, $[L^r(Q)]^*=L^{r'}(Q)$,
we easily find that, for any $r\in [1,\infty]$
and any invertible matrix $A$, $[L^r(Q,A)]^*=L^{r'}(Q,A^{-1})$.
From this and \eqref{wt}, we infer that there exists $\vec h_{Q,0}\in
L^{q}(Q,A_Q^{-1})$ such that, for any $\vec f\in L^{q'}_s(Q,A_Q)\subset L^{q'}(Q,A_Q)$,
\begin{align}\label{csk}
L\left(\vec f\right)=\int_Q\vec f(x)\vec h_{Q,0}(x)\,dx.
\end{align}
Now, we claim that $\vec h_{Q,0}$ is unique in the sense of $L^{q}(Q,A_Q^{-1})/
[\mathcal{P}_s(Q)]^m$,
that is, if there exists $\vec h_{Q,0}'\in L^{q}(Q,A_Q^{-1})$ satisfying, for any
$\vec f\in L^{q'}_s(Q,A_Q)$, \eqref{csk} holds,
then $\vec h_{Q,0}-\vec h_{Q,0}'\in[\mathcal{P}_s(Q)]^m$, where
$\mathcal{P}_s(Q):=\mathcal{P}_s|_Q$ (the restriction to $Q$ of $\mathcal{P}_s$).
By \eqref{csk}, \eqref{csk} with $\vec h_{Q,0}$ replaced by $\vec h_{Q,0}'$,
the definition of $L^{q'}_s(Q,A_Q)$, and \eqref{fy},
we find that, for any $\vec f\in L^{q'}_s(Q,A_Q)$,
\begin{align}\label{r3}
0&=\int_Q\left[\vec f(x)-\Pi_Q^s\left(\vec f\right)(x)\right]\left[
\vec h_{Q,0}(x)-\vec h_{Q,0}'(x)\right]\,dx\nonumber\\
&=\int_Q\vec f(x)\left[\vec h_{Q,0}(x)-\vec h_{Q,0}'(x)\right]\,dx
-\int_Q\Pi_Q^s\left(\vec f\right)(x)\Pi_Q^s
\left(\vec h_{Q,0}-\vec h_{Q,0}'\right)(x)\,dx\nonumber\\
&=\int_Q\vec f(x)\left[\vec h_{Q,0}(x)-\vec h_{Q,0}'(x)-\Pi_Q^s
\left(\vec h_{Q,0}-\vec h_{Q,0}'\right)(x)\right]\,dx.
\end{align}
From the definition of $L^{q'}_s(Q,A_Q)$, we deduce that, for any
$\vec g\in L^{q'}(Q,A_Q)$,
$\vec g-\Pi^s_Q(\vec g)\in L^{q'}_s(Q,A_Q)$, which, together with \eqref{fy}
and \eqref{r3}, further implies that
\begin{align*}
&\int_Q\vec g(x)\left[\vec h_{Q,0}(x)-\vec h_{Q,0}'(x)-\Pi_Q^s
\left(\vec h_{Q,0}-\vec h_{Q,0}'\right)(x)\right]\,dx\\
&\quad=\int_Q\left[\vec g(x)-\Pi^s_Q\left(\vec g\right)(x)\right]\left[\vec h_{Q,0}(x)-\vec h_{Q,0}'(x)-\Pi_Q^s
\left(\vec h_{Q,0}-\vec h_{Q,0}'\right)(x)\right]\,dx\\&\quad=0
\end{align*}
From the arbitrariness of $\vec g\in L^{q'}(Q,A_Q)$, we infer that
$\vec h_{Q,0}-\vec h_{Q,0}'=\Pi_Q^s(\vec h_{Q,0}-\vec h_{Q,0}')$ almost everywhere
and hence $\vec h_{Q,0}-\vec h_{Q,0}'\in[\mathcal{P}_s(Q)]^m$. This shows the above claim holds.
For any cube $Q$, let $\vec h_Q\in L^{q}(Q,A_Q^{-1})/
[\mathcal{P}_s(Q)]^m$ be such that $\vec h_{Q,0}\in\vec h_Q$.
By the uniqueness, we conclude that, for any cubes $Q_1,Q_2$,
any $\vec h_{Q_1,0}\in\vec h_{Q_1}$, and any $\vec h_{Q_2,0}\in\vec h_{Q_2}$,
$$\vec h_{Q_1,0}|_{Q_1\cap Q_2}-\vec h_{Q_2,0}|_{Q_1\cap Q_2}\in
\left[\mathcal{P}_s\left(Q_1\cap Q_2\right)\right]^m.$$
From this, it follows that there exists a measurable function
$\vec h$ such that,
for any cube $Q$, $\vec h\mathbf1_{Q}\in\vec h_Q$.
By the definition of
$H^{p,{q'},s}_{W,\rm{fin}}$, we find that, for any $\vec f\in H^{p,{q'},s}_{W,\rm{fin}}$,\
there exists cube $Q$ such that $\operatorname{supp}\vec f\subset Q$ and
$\vec f\in H^{p,{q'},s}_{W,\rm{fin}}\subset L^{q'}_s(Q,A_Q)$, which, together with
\eqref{csk}, further implies that
\begin{align*}
L\left(\vec f\right)=\int_{\mathbb R^n}\vec f(x)\vec h(x)\,dx.
\end{align*}
Using this and \eqref{wt}, we conclude that, for any cube $Q$,
\begin{align}\label{h<L}
\left\|\vec h\right\|_{[L^{q'}_s(Q,A_Q)]^*}\lesssim
|Q|^{-\frac1{q'}+\frac1p}\|L\|_{(H^p_W)^*},
\end{align}
where the implicit positive constant is independent of $Q$.
By an argument similar to the proof of \cite[(8.12)]{b03},
we find that
\begin{align*}
\left\|\vec h\right\|_{[L^{q'}_s(Q,A_Q)]^*}=\inf_{\vec P\in(\mathcal{P}_s)^m}
\left\|\vec h-\vec P\right\|_{L^{q}_s(Q,A_Q^{-1})}.
\end{align*}
Combining this, the definition of $\mathcal L_{p,q,s,W}$, and
\eqref{h<L} yields
\begin{align*}
\left\|\vec h\right\|_{\mathcal L_{p,q,s,W}}:=&\,
\sup_{\mathrm{cube}\,Q}\inf_{\vec P\in(\mathcal{P}_s)^m}|Q|^{\frac1{q'}-\frac1p}
\left\|\vec h-\vec P\right\|_{L^{q}(Q,A_Q^{-1})}\\
=&\,
\sup_{\mathrm{cube}\,Q}|Q|^{\frac1{q'}-\frac1p}
\left\|\vec h\right\|_{[L^{q'}(Q,A_Q)]^*}
\lesssim\|L\|_{(H^p_W)^*}.
\end{align*}
This finishes the proof of (ii) and hence Theorem \ref{dual}.
\end{proof}

As an immediate consequence of Theorem \ref{dual},
we have the following equivalent characterization
of the matrix-weighted Campanato space; we omit the details.

\begin{corollary}\label{2247}
Let $p\in(0,1]$, $q\in[1,\infty)$, $s\in
[\lfloor n(\frac{1}{p}-1)\rfloor,\infty)\cap\mathbb Z_+$, and $W\in A_p$.
Then $\mathcal L_{p,1,s_0,W}=\mathcal L_{p,q,s,W}$ with equivalent quasi-norms, where
$s_0:=\lfloor n(\frac{1}{p}-1)\rfloor$.
\end{corollary}

Next, we present another equivalent characterization
of the matrix-weighted Campanato space.

\begin{theorem}\label{JN}
Let $p\in(0,1]$, $q\in[1,\infty)$, $s\in
[\lfloor n(\frac{1}{p}-1)\rfloor,\infty)\cap\mathbb Z_+$,
and $r\in[1,\infty)$.
Let $W\in A_p$ and $\{A_Q\}_{\mathrm{cube}\,Q}$ be a family of
reducing operators of order $p$ for $W$.
Then $\vec f\in\mathcal L_{p,q,s,W}$ if and only if
\begin{align}\label{JNsim}
\left\|\vec f\right\|_{\mathcal L_{p,q,s,W,\mathbb A}}:=\sup_{\mathrm{cube}\,Q}\inf_{\vec P\in(\mathcal{P}_s)^m}|Q|^{\frac1{q'}-\frac1p}
\int_{Q}\left|W^{-\frac1{r'}}(x)A_Q^{\frac p{r'}-1}\left[\vec g(x)-\vec P(x)\right]
\right|^r\,dx<\infty
\end{align}
with equivalent quasi-norms.
\end{theorem}

\begin{proof}
Using the H\"older inequality, \cite[Lemma 2.8]{byymz},
\cite[Lemma 2.9]{byyz},
and Lemma \ref{86}, we conclude that, for any cube $Q$ and $\vec P
\in(\mathcal{P}_s)^m$,
\begin{align*}
&\int_{Q}\left|A_Q^{-1}\left[\vec g(x)-\vec P(x)\right]
\right|\,dx\\
&\quad\leq\left[\int_{Q}\left\|A_Q^{-\frac p{r'}}
W^{\frac1{r'}}(x)\right\|^{r'}\,dx\right]^{\frac1{r'}}
\left\{\int_{Q}\left|W^{-\frac1{r'}}(x)A_Q^{\frac p{r'}-1}\left[\vec g(x)-\vec P(x)\right]
\right|^r\,dx\right\}^{\frac1r}\\
&\quad\leq\left[\int_{Q}\left\|W^{\frac1p}(x)A_Q^{-1}\right\|^p\,dx\right]^{\frac1{r'}}
\left\{\int_{Q}\left|W^{-\frac1{r'}}(x)A_Q^{\frac p{r'}-1}\left[\vec g(x)-\vec P(x)\right]
\right|^r\,dx\right\}^{\frac1r}\\
&\quad\lesssim
\left\{\int_{Q}\left|W^{-\frac1{r'}}(x)A_Q^{\frac p{r'}-1}\left[\vec g(x)-\vec P(x)\right]
\right|^r\,dx\right\}^{\frac1r}.
\end{align*}
From \cite[Lemma 2.9]{byyz} and \cite[Lemma 3.3]{fr21}, we deduce that
\begin{align*}
&\int_{Q}\left|W^{-\frac1{r'}}(x)A_Q^{\frac p{r'}-1}\left[\vec g(x)-\vec P(x)\right]
\right|^r\,dx\\
&\quad\leq\int_{Q}\left\|A_Q^{\frac p{r'}}W^{-\frac1{r'}}(x)\right\|^r
\left|A_Q^{-1}\left[\vec g(x)-\vec P(x)\right]
\right|^r\,dx\\
&\quad\leq\int_{Q}\left\|A_QW^{-\frac1p}(x)\right\|^{p(r-1)}
\left|A_Q^{-1}\left[\vec g(x)-\vec P(x)\right]
\right|^r\,dx\\
&\quad\lesssim\int_{Q}\left|A_Q^{-1}\left[\vec g(x)-\vec P(x)\right]
\right|^r\,dx.
\end{align*}
These, together with Corollary \ref{2247}, further finish
the proof of Theorem \ref{JN}.
\end{proof}

\begin{remark}
Let $m=1$, $p=q=1$, $s=0$, $r\in[1,\infty)$, and $W\in A_1$.
Then \eqref{JNsim} reduces to
\begin{align}\label{bxhl}
\sup_{\mathrm{cube}\,Q}\inf_{\vec P\in\mathcal{P}_0}
\frac1{W(Q)}\int_{Q}\left|W^{-\frac1{r'}}(x)
\left[g(x)-P(x)\right]\right|^r\,dx<\infty.
\end{align}
In this case, Theorem \ref{JN} reduces to
$g\in\mathcal L_{1,1,0,W}$
if and only if \eqref{bxhl} holds, which, together with
Lemma \ref{lPi}, further implies that
Theorem \ref{JN} coincides with \cite[Theorem 4]{MW75} in the case of $p=1$.
\end{remark}

\section{Boundedness of Calder\'on--Zygmund Operators
on $\mathcal L_{p,q,s,W}$\label{cz-l}}

In this section, we establish the boundedness
of Calder\'on--Zygmund operators on $\mathcal L_{p,q,s,W}$.
We first present the concept of the $(s,\delta)$-standard kernel
(see, for instance, \cite[Chapter III]{S93}).
In what follows, for any $\gamma=(\gamma_1,\ldots,\gamma_n)\in
{\mathbb Z}_+^n$,
any $\gamma$-order differentiable function $F(\cdot,\cdot)$
on ${\mathbb{R}^n}\times {\mathbb{R}^n}$, and any $(x,y)\in {\mathbb{R}^n}\times {\mathbb{R}^n}$, let
$$\partial_{(1)}^{\gamma}F(x,y):=\frac{\partial^{|\gamma|}}
{\partial x_1^{\gamma_1}\cdots\partial x_n^{\gamma_n}}F(x,y)$$
and
$$\partial_{(2)}^{\gamma}F(x,y):=\frac{\partial^{|\gamma|}}
{\partial y_1^{\gamma_1}\cdots\partial y_n^{\gamma_n}}F(x,y).$$

\begin{definition}\label{skernel}
Let $s\in{\mathbb Z}_+$ and $\delta\in(0,1]$. A measurable function $K$
on ${\mathbb{R}^n}\times {\mathbb{R}^n}\setminus\{(x,x):  x\in{\mathbb{R}^n}\}$
is called an \emph{$(s,\delta)$-standard kernel} if
there exists a positive constant $C$
such that, for any $\gamma\in{\mathbb Z}_+^n$ with $|\gamma|\le s$,
the followings hold:
\begin{itemize}
\item[\rm (i)]
for any $x,y\in{\mathbb{R}^n}$ with $x\neq y$,
\begin{align}\label{size-s'}
\left|\partial_{(2)}^{\gamma}K(x,y)\right|\le
\frac{C}{|x-y|^{n+|\gamma|}};
\end{align}

\item[\rm (ii)]
\eqref{size-s'} still holds for the first variable of $K$;

\item[\rm (iii)]
for any $x,y,z\in{\mathbb{R}^n}$
with $x\neq y$ and $|x-y|\ge2|y-z|$,
\begin{align}\label{regular2-s}
\left|\partial_{(2)}^{\gamma}K(x,y)-\partial_{(2)}^{\gamma}K(x,z)\right|
\le C\frac{|y-z|^\delta}{|x-y|^{n+|\gamma|+\delta}};
\end{align}

\item[\rm (iv)]
\eqref{regular2-s} still holds for the first variable of $K$.
\end{itemize}
\end{definition}

\begin{definition}\label{scz}
Let $s\in\mathbb Z_+$, $\delta\in(0,1]$, and $K$ be an $(s,\delta)$-standard kernel.
A linear operator $T$ is called an $(s,\delta)$-\emph{Calder\'on--Zygmund operator},
associated with the kernel $K$
if $T$ is bounded on $L^2$ and,
for any $f\in L^2$ with compact support
and for almost every $x\not\in\mathrm{\,supp\,}(f)$,
\begin{align}\label{eqscz}
T(f)(x)=\int_{\mathbb{R}^n}K(x,y)f(y)\,dy.
\end{align}
\end{definition}

The following proposition shows that
any Calder\'on--Zygmund operator
is bounded on $L^q$ for any
$q\in(1,\infty)$, which is exactly \cite[Theorem 5.10]{Duo01}.

\begin{proposition}\label{sczbouned}
Let $s\in\mathbb Z_+$, $\delta\in(0,1]$,
$K$ be an $(s,\delta)$-standard kernel, and
$T$ an $(s,\delta)$-Calder\'on--Zygmund operator associated
with the kernel $K$. If $q\in(1,\infty)$,
then there exists $C\in(0,\infty)$
such that, for any $f\in L^q$,
$$\|T(f)\|_{L^q}\le C\|f\|_{L^q}.$$
\end{proposition}

\begin{remark}\label{lq}
Let $s\in\mathbb Z_+$, $\delta\in(0,1]$,
$K$ be an $(s,\delta)$-standard kernel, and
$T$ an $(s,\delta)$-Calder\'on--Zygmund operator associated
with the kernel $K$.
\begin{enumerate}
\item[{\rm(i)}] By \cite[TH\'EOR\`{E}ME 21]{cm78},
we find that, for any $q\in(1,\infty)$, any
$f\in L^q$, and almost every
$x\not\in\mathrm{\,supp\,}(f)$, \eqref{eqscz} also holds.
\item[{\rm(ii)}] Let $T^*$ denote the adjoint operator of
$T$ on $L^2$,
that is, for any $f,g\in L^2$,
\begin{align}\label{71}
\int_{\mathbb R^n}T^*f(x)g(x)\,dx=\int_{\mathbb R^n}f(x)Tg(x)\,dx.
\end{align}
Then, from Definition \ref{scz}, we easily
infer that $T^*$ is an $(s,\delta)$-Calder\'on--Zygmund operator associated
with the kernel $\widetilde{K}$, where, for any
$x,y\in\mathbb{R}^n$ with $x\ne y$,
$\widetilde{K}(x,y):={K(y,x)}$.
Moreover, by this and Proposition \ref{sczbouned},
we find that \eqref{71} also holds for any
$p\in(1,\infty)$, $f\in L^p$, and $g\in L^{p'}$.
\end{enumerate}
\end{remark}

Now, we recall
the definition of
modified Calder\'on--Zygmund operators as follows;
see, for instance, \cite[p.\,119]{Duo01} or \cite[Section 4.4]{kk13}.

\begin{definition}\label{mcz}
Let $s\in{\mathbb Z}_+$ and $\delta\in(0,1]$.
Assume that $K$ is an $(s,\delta)$-standard kernel and
$T$ an $(s,\delta)$-Calder\'on--Zygmund operator associated with the kernel $K$.
Let $\{t_i\}_{i\in\mathbb{N}}$ be a sequence of $(0,\infty)$ such that,
for any $i\in\mathbb{N}$, $t_{i}\le t_{i+1}$,
$\lim_{i\to\infty}t_i=\infty$,
and $\{B_i\}_{i\in\mathbb{N}}
:=\{B({\bf0},t_i)\}_{i\in\mathbb{N}}$.
We also let $B_0:=\emptyset$.
The \emph{$(s,\delta)$-modified Calder\'on--Zygmund operator}
$T^{(s)}_{\{t_i\}_{i\in\mathbb{N}}}$,
associated with $K$ and
$\{t_i\}_{i\in\mathbb{N}}$, is defined by setting,
for any suitable function $f$ on $\mathbb{R}^n$
and for almost every $x\in\mathbb{R}^n$,
\begin{align*}
T^{(s)}_{\{t_i\}_{i\in\mathbb{N}}}(f)(x)&:=\sum_{i\in\mathbb{N}}
{\bf1}_{B_i\setminus B_{i-1}}(x)
\Bigg\{T\left({\bf1}_{2B_{i}}f\right)(x)\\
&\qquad\left.
+\int_{\mathbb{R}^n\setminus 2B_{i}}\left[K(x,y)-
\sum_{\{\gamma\in\mathbb Z^n_+:  |\gamma|\leq s\}}
\frac{\partial^\gamma_{(1)}K({\bf0},y)}{\gamma!}x^\gamma
\right]f(y)\,dy\right.\nonumber\\
&\qquad\left.-\int_{2B_i\setminus 2B_1}
\sum_{\{\gamma\in\mathbb Z^n_+:  |\gamma|\leq s\}}
\frac{\partial^\gamma_{(1)}K({\bf0},y)}{\gamma!}x^\gamma
f(y)\,dy\right\}.
\end{align*}
\end{definition}

The following proposition shows that
the $(s,\delta)$-modified Calder\'on--Zygmund operator
is well-defined for appropriate locally
integrable functions, whose proof is
a slight modification of the proof of \cite[Proposition 3.6]{jly};
we omit the details.

\begin{proposition}\label{mzcwell}
Let $s\in\mathbb Z_+$, $\delta\in(0,1]$, $K$ be an $(s,\delta)$-standard kernel as in
Definition \ref{skernel}, and
$T$ an $(s,\delta)$-Calder\'on--Zygmund operator associated with the kernel $K$.
Assume that $\{t_i\}_{i\in\mathbb{N}}$
and $T^{(s)}_{\{t_i\}_{i\in\mathbb{N}}}$ are as in Definition \ref{mcz}.
If $q\in(1,\infty)$ and $f\in L^q_{\rm loc}$
satisfies
\begin{align}\label{locq}
\int_{|y|\ge2t_1}\frac{|f(y)|}{|y|^{n+s+\delta}}\,dy<\infty,
\end{align}
then
\begin{enumerate}
\item[$\mathrm{(i)}$]
$T^{(s)}_{\{t_i\}_{i\in\mathbb{N}}}(f)$ is finite almost everywhere on $\mathbb{R}^n$;

\item[$\mathrm{(ii)}$] if $f\in L^q$, then
$T(f)-T^{(s)}_{\{t_i\}_{i\in\mathbb{N}}}(f)\in\mathcal{P}_s$;

\item[$\mathrm{(iii)}$]
if $\{s_j\}_{j\in\mathbb{N}}$
is another sequence of $(0,\infty)$ satisfying the same
conditions as $\{t_i\}_{i\in\mathbb{N}}$,
then $$T^{(s)}_{\{t_i\}_{i\in\mathbb{N}}}(f)-
T^{(s)}_{\{s_j\}_{j\in\mathbb{N}}}(f)\in\mathcal{P}_s.$$
\end{enumerate}
\end{proposition}

As an application of Proposition
\ref{mzcwell}, we now prove that
the modified Calder\'on--Zygmund operator is also
well-defined
on certain matrix-weighted Campanato spaces.

\begin{theorem}\label{sczldy}
Let $p\in(0,\infty)$, $q\in[1,\infty)$, $s\in
[\lceil n(\frac{1}{p}-1)\rceil,\infty)\cap\mathbb Z_+$, $\delta\in(0,1]$,
$W\in A_p$, and $K$ be an $(s,\delta)$-standard kernel.
If $\{t_i\}_{i\in\mathbb{N}}$
and $T^{(s)}_{\{t_i\}_{i\in\mathbb{N}}}$ are as in Definition \ref{mcz},
then
\begin{enumerate}
\item[$\mathrm{(i)}$]
for any $\vec f\in\mathcal L_{p,q,s,W}$,
$T^{(s)}_{\{t_i\}_{i\in\mathbb{N}}}(\vec f)$ is finite
almost everywhere on $\mathbb{R}^n$;

\item[$\mathrm{(ii)}$]
if $\{s_j\}_{j\in\mathbb{N}}$
is another sequence of $(0,\infty)$ satisfying the same
conditions as $\{t_i\}_{i\in\mathbb{N}}$,
then, for any $\vec f\in\mathcal L_{p,q,s,W}$,
$T^{(s)}_{\{t_i\}_{i\in\mathbb{N}}}(\vec f)
-T^{(s)}_{\{s_j\}_{j\in\mathbb{N}}}(\vec f)=\vec P$
with $\vec P\in(\mathcal{P}_s)^m$.
\end{enumerate}
\end{theorem}

\begin{proof}
Let $\vec f\in\mathcal L_{p,q,s,W}$.
To prove the present theorem,
by using Proposition \ref{mzcwell}, we
conclude that
we only need to show that $\vec f$ satisfies \eqref{locq}
with $f$ therein replaced by $\vec f$.

To this end, we first establish a useful estimate, which will
be used repeatedly later. Indeed, from the definition of $\mathcal L_{p,q,s,W}$,
we deduce that $\vec f\in (L^q_{\rm loc})^m$.
Let $Q$ be a cube centered at $x\in\mathbb R^n$ with edge length $t\in(0,\infty)$.
By \cite[Lemma 2.20]{jtyyz22i}, the H\"older inequality, Lemma \ref{p01},
and the definition of $\mathcal L_{p,q,s,W}$,
we find that, for any $\vec f\in\mathcal L_{p,q,s,W}$,
\begin{align}\label{xz}
&\int_{{\mathbb{R}^n}\setminus Q}
\frac{|\vec f(y)-\Pi_Q^s(\vec f)(y)|}
{|x-y|^{n+s+\delta}}\,dy\nonumber\\
&\quad\lesssim
\sum_{k\in\mathbb{N}}
\left(2^kt\right)^{-s-{\delta}}
\left[\fint_{2^kQ}
\left|\vec f(y)-\Pi_{Q}^s\left(\vec f\right)(y)\right|^q
\,dy\right]^\frac{1}{q}\nonumber\\
&\quad\lesssim
\sum_{k\in\mathbb{N}}
\left(2^kt\right)^{-s-{\delta}}
\left[\fint_{2^kQ}
\left|\vec f(y)-\Pi_{2^kQ}^s\left(\vec f\right)(y)\right|^q
\,dy\right]^\frac{1}{q}\nonumber\\
&\quad\le
\sum_{k\in\mathbb{N}}
\left(2^kt\right)^{-s-{\delta}}\left\|A_{2^kQ}A_Q^{-1}\right\|
\left\|A_Q\right\|
\left\{\fint_{2^kQ}
\left|A_{2^kQ}^{-1}\left[\vec f(y)
-\Pi_{2^kQ}^s\left(\vec f\right)(y)\right]\right|^q
\,dy\right\}^\frac{1}{q}\nonumber\\
&\quad\lesssim\sum_{k\in\mathbb{N}}\left(2^kt\right)^{-s-{\delta}}
2^{k\widetilde{d}(1-\frac{1}{p})}\left\|A_Q\right\|\left|2^kQ\right|^{-1+\frac1{p}}
\left\|\vec f\right\|_{\mathcal L_{p,q,s,W}}\nonumber\\
&\quad\lesssim t^{-s-\delta-n(1-\frac 1{p})}\left\|A_Q\right\|
\left\|\vec f\right\|_{\mathcal L_{p,q,s,W}},
\end{align}
where $\widetilde{d}$ is the same as in Lemma \ref{p01}.

Let $Q_0$ be a cube centered at $\mathbf 0$ with edge length $2n^{-\frac12}t_1$.
Since $\vec f\in (L^q_{\rm{loc}})^m\subset (L^1_{\rm{loc}})^m$, we may
assume that, for any $i\in\{1,\ldots,m\}$ and $x\in\mathbb R^n$,
$$\left[\Pi^s_{Q_0}\left(\vec f\right)\right]_i(x):=\sum_{\{\gamma\in\mathbb Z^n_+:  |\gamma|\leq s\}}
c_{\gamma,i} x^\gamma.$$
Applying \eqref{xz} and the H\"older inequality, we obtain
\begin{align}\label{dtwhs}
&\int_{|y|\ge2t_1}\frac{|f(y)|}{|y|^{n+s+\delta}}\,dy\nonumber\\
&\quad\le\int_{Q_0^\complement}\frac{|f(y)-\Pi^s_{Q_0}(\vec f)(y)|}{|y|^{n+s+\delta}}\,dy
+\int_{Q_0^\complement}\frac{|\Pi^s_{Q_0}(\vec f)(y)|}{|y|^{n+s+\delta}}\,dy\nonumber\\
&\quad\lesssim t_1^{-s-\delta-n(1-\frac 1{p})}\left\|A_{Q_0}\right\|
\left\|\vec f\right\|_{\mathcal L_{p,q,s,W}}
+\int_{Q_0^\complement}\sum_{i=1}^m\sum_{\{\gamma\in\mathbb Z^n_+:  |\gamma|\leq s\}}
|c_{\gamma,i}|\,|y|^{|\gamma|-n-s-\delta}\,dy\nonumber\\&\quad<\infty.
\end{align}
This further implies that $\vec f$ satisfies \eqref{locq} and
hence finishes the proof of Theorem \ref{sczldy}.
\end{proof}

By Theorem \ref{sczldy}, we
find that, in the matrix-weighted Campanato space,
the operator $T^{(s)}_{\{t_j\}_{j\in\mathbb{N}}}$
is independent of the choice of
$\{t_j\}_{j\in\mathbb{N}}$;
therefore, in what follows, we simply write
$T^{(s)}_{\{t_j\}_{j\in\mathbb{N}}}$
as $\widetilde{T}$.
Next, we establish the boundedness characterization
of $\widetilde{T}$ on
matrix-weighted Campanato spaces.
To this end, recall
that a Calder\'on--Zygmund operator $T$
is said to satisfy the \emph{vanishing conditions up to order $s$}
if, for any $\gamma\in\mathbb Z_+^n$ with $|\gamma|\leq s$,
$T(x^\gamma)=0$, which means that, for any $a\in L^2$ satisfying that
$\int_{{\mathbb{R}^n}} a(x)x^\gamma\,dx=0$
with compact support,
\begin{align*}
\int_{\mathbb{R}^n}T^*(a)(x)x^\gamma\,dx=0,
\end{align*}
where $T^*$ denotes the adjoint operator of
$T$ on $L^2$.
Based on this, we obtain the following
boundedness characterization
of $\widetilde{T}$ on
matrix-weighted Campanato spaces.

\begin{theorem}\label{btbbmc}
Let $p\in(0,\infty)$, $q\in[1,\infty)$, $\delta\in(0,1]$, $s\in
[\lceil n(\frac{1}{p}-1)\rceil,\infty)\cap\mathbb Z_+$, and $W\in A_p$.
Assume that $K$ is an $(s,\delta)$-standard kernel as in Definition \ref{skernel}
and $\widetilde{T}$ an $(s,\delta)$-modified
Calder\'on--Zygmund operator associated with the kernel $K$.
Then $\widetilde{T}$ is bounded on $\mathcal L_{p,q,s,W}$
if and only if, for any $\gamma\in\mathbb Z_+^n$ with $|\gamma|\leq s$,
$T(x^\gamma)=0$.
\end{theorem}

To prove this theorem, we need the following two lemmas.

\begin{lemma}\label{fgduiou}
Let $p\in(0,\infty)$, $q\in[1,\infty)$, $\delta\in(0,1]$,
$s\in[\lceil n(\frac{1}{p}-1)\rceil,\infty)\cap\mathbb Z_+$, and $W\in A_p$.
Assume that $K$ is an $(s,\delta)$-standard kernel as in Definition \ref{skernel}.
If $T$ is an $(s,\delta)$-Calder\'on--Zygmund operator associated with the kernel $K$
and $\widetilde{T}$ the corresponding $(s,\delta)$-modified one as in Definition \ref{mcz},
then, for any $\vec f\in (L^{q'}_s)^m$ with compact
support and
for any $\vec g\in\mathcal L_{p,q,s,W}$,
\begin{align*}
\int_{\mathbb{R}^n}T^*\left(\vec f\right)(x)\vec g(x)\,dx
=\int_{\mathbb{R}^n}\vec f(x)\widetilde{T}\left(\vec g\right)(x)\,dx.
\end{align*}
\end{lemma}

\begin{proof}
Assume that $\vec f\in [L^{q'}_0(B({\bf0},t_1))]^m$
with $t_1\in(0,\infty)$
and fix $\vec g\in\mathcal L_{p,q,s,W}$.
Choose a sequence $\{t_j\}_{j\in\mathbb{N}\cap[2,\infty)}$ of $(0,\infty)$
satisfying that, for any $j\in\mathbb{N}$, $t_{j}\le t_{j+1}$
and $\lim_{j\to\infty}t_j=\infty$.
Define $B_1:=B({\bf0},t_1)$.
Then, using the definition of $\mathcal L_{p,q,s,W}$,
we conclude that $\vec g{\bf 1}_{2B_1}\in (L^q)^m$, which, together
with Remark \ref{lq}(ii), further implies that
\begin{align}\label{fgduiou-e1}
\int_{\mathbb{R}^n}T^*\left(\vec f\right)(y)\vec g(y)\,dy
&=\int_{2B_1}T^*\left(\vec f\right)(y)\vec g(y)\,dy
+\int_{\mathbb{R}^n\setminus 2B_1}T^*
\left(\vec f\right)(y)\vec g(y)\,dy\nonumber\\
&=\int_{\mathbb{R}^n}\vec f(y)T\left(\vec g{\bf1}_{2B_1}\right)(y)\,dy\nonumber\\
&\quad+\int_{\mathbb{R}^n\setminus 2B_1}\vec g(y)
\int_{\mathbb{R}^n}K(x,y)\vec f(x)\,dx\,dy.
\end{align}
By the fact that $\mathrm{\,supp\,}(\vec f)\subset B_1$, Definition \ref{skernel}(iv),
Tonelli's theorem, the H\"older inequality, and \eqref{dtwhs},
we find that
\begin{align*}
&\int_{\mathbb{R}^n\setminus 2B_1}\left|\vec g(y)\right|
\int_{\mathbb{R}^n}\left|K(x,y)-\sum_{\{\gamma\in\mathbb Z^n_+:  |\gamma|\leq s\}}
\frac{\partial^\gamma_{(1)}K({\bf0},y)}{\gamma!}x^\gamma\right|\left|\vec f(x)\right|\,dx\,dy\\
&\quad=\int_{\mathbb{R}^n\setminus 2B_1}\left|\vec g(y)\right|
\int_{\mathbb{R}^n}\left|\sum_{\{\gamma\in\mathbb Z^n_+:  |\gamma|=s\}}
\frac{\partial^\gamma_{(1)}K({\xi_x},y)-\partial^\gamma_{(1)}K({\bf0},y)}{\gamma!}x^\gamma\right|
\left|\vec f(x)\right|\,dx\,dy\\
&\quad\lesssim\int_{\mathbb{R}^n\setminus 2B_1}\left|\vec g(y)\right|
\int_{B_1}\frac{|x|^{s+\delta}}{|y|^{n+s+\delta}}\left|\vec f(x)\right|\,dx\,dy\\
&\quad\le t_1^{s+\delta}\int_{\mathbb{R}^n\setminus 2B_1}
\frac{|\vec g(y)|}{|y|^{n+s+\delta}}\,dy
\int_{B_1}\left|\vec f(x)\right|\,dx\\
&\quad\le t_1^{s+\delta}\left|B_1\right|^{1-\frac{1}{q}}\left\|\vec f\right\|_{L^{q'}(B_1)}
\int_{\mathbb{R}^n\setminus 2B_1}\frac{|\vec g(y)|}{|y|^{n+s+\delta}}\,dy
<\infty,
\end{align*}
where $\xi_x=\theta x$ with $\theta\in [0,1]$.
From this, $\vec f\in (L^{q'}_s)^m$, and Fubini's theorem,
we deduce that
\begin{align*}
&\int_{\mathbb{R}^n\setminus 2B_1}\vec g(y)
\int_{\mathbb{R}^n}K(x,y)\vec f(x)\,dx\,dy\\
&\quad=\int_{\mathbb{R}^n\setminus 2B_1}\vec g(y)
\int_{\mathbb{R}^n}\left[K(x,y)-\sum_{\{\gamma\in\mathbb Z^n_+:  |\gamma|\leq s\}}
\frac{\partial^\gamma_{(1)}K({\bf0},y)}{\gamma!}x^\gamma\right]\vec f(x)\,dx\,dy\\
&\quad=\int_{\mathbb{R}^n}\vec f(x)
\left\{\int_{\mathbb{R}^n\setminus 2B_1}
\left[K(x,y)-\sum_{\{\gamma\in\mathbb Z^n_+:  |\gamma|\leq s\}}
\frac{\partial^\gamma_{(1)}K({\bf0},y)}{\gamma!}x^\gamma\right]\vec g(y)\,dy\right\}\,dx.
\end{align*}
Combining this, \eqref{fgduiou-e1}, Definition \ref{mcz}, and Theorem \ref{sczldy},
we obtain
\begin{align*}
\int_{\mathbb{R}^n}T^*\left(\vec f\right)(x)\vec g(x)\,dx
=\int_{\mathbb{R}^n}\vec f(x)
T_{\{t_i\}_{i\in\mathbb{N}}}\left(\vec g\right)(x)\,dx
=\int_{\mathbb{R}^n}\vec f(x)\widetilde{T}\left(\vec g\right)(x)\,dx,
\end{align*}
which, together with the arbitrariness of $t_1$, $\vec f$, and $\vec g$,
completes the proof of Lemma \ref{fgduiou}.
\end{proof}

Recall that, for any $s\in\mathbb Z_+$, any $p\in(1,\infty]$, and any
compact set $E\subset\mathbb R^n$,
\begin{align*}
L^p_{s}(E):=\left\{\vec f\in L^p(E):  \Pi^s_{E}\left(
\vec f\right)=\vec 0\right\}.
\end{align*}

\begin{lemma}\label{2501132155}
Let $s\in\mathbb Z_+$, $\delta\in(0,1]$, $K$ be an $(s,\delta)$-standard kernel as in
Definition \ref{skernel},
$T$ an $(s,\delta)$-Calder\'on--Zygmund operator associated with the kernel $K$,
and $\widetilde{T}$ the corresponding $(s,\delta)$-modified
one as in Definition \ref{mcz}.
Then, for any $\gamma\in\mathbb Z_+^n$ with $|\gamma|\leq s$,
$\widetilde{T}(x^\gamma)\in\mathcal{P}_s$
if and only if $T(x^\gamma)=0$.
\end{lemma}

\begin{proof}
We first show the necessity.
Assume that, for any $\gamma\in\mathbb Z_+^n$ with $|\gamma|\leq s$,
$\widetilde{T}(x^\gamma)\in\mathcal{P}_s$.
Using this and Lemma \ref{fgduiou},
we conclude that, for any $a\in L^2_s(E)$ with
compact support $E\subset\mathbb R^n$,
\begin{align*}
\int_{\mathbb{R}^n}T^*(a)(x)x^\gamma\,dx
=\int_{\mathbb{R}^n}\widetilde{T}(\cdot^\gamma)(x)a(x)\,dx
=0,
\end{align*}
which further implies that $T(x^\gamma)=0$.
This finishes the proof of the necessity.

Now, we show the sufficiency.
To this end, fix $\gamma\in\mathbb Z_+^n$ with $|\gamma|\leq s$, assume that
$T(x^\gamma)=0$, and choose a
ball $B\subset\mathbb{R}^n$ and a function $h\in L^2(B)$.
Therefore, $h-\Pi^s_B(h)\in L^2_s(B)$.
Combining this, Lemma \ref{fgduiou},
and the fact that
$T(x^\gamma)=0$, we find that
\begin{align*}
&\int_{\mathbb{R}^n}\left[\widetilde{T}(\cdot^\gamma)(x)-
\Pi_B^s\left(\widetilde{T}(\cdot^\gamma)\right)\right]h(x)\,dx\\
&\quad=\int_{\mathbb{R}^n}\left[\widetilde{T}(\cdot^\gamma)(x)-
\Pi_B^s\left(\widetilde{T}(\cdot^\gamma)\right)\right]
\left[h(x)-\Pi^s_B(h)(x)\right]\,dx\\
&\quad=\int_{\mathbb{R}^n}\widetilde{T}(\cdot^\gamma)(x)\left[
h(x)-\Pi^s_B(h)(x)\right]\,dx\\
&\quad=\int_{\mathbb{R}^n}T^*\left(h-\Pi^s_B(h)\right)(x)x^\gamma\,dx
=0.
\end{align*}
From this and the arbitrariness of $h\in L^2(B)$,
it follows that $\widetilde{T}(x^\gamma)=
\Pi^s_B(\widetilde{T}(x^\gamma))$ almost everywhere in $B$.
Applying this and the arbitrariness of the ball $B$,
we conclude that $\widetilde{T}(x^\gamma)\in\mathcal{P}_s$.
This finishes the proof of the sufficiency and hence
Lemma \ref{2501132155}.
\end{proof}

Next, we prove Theorem \ref{btbbmc}.

\begin{proof}[Proof of Theorem \ref{btbbmc}]
We first show the necessity.
Assume that $\widetilde{T}$ is bounded on
$\mathcal L_{p,q,s,W}$.
Then, by this, we find that, for any $\vec q\in(\mathcal{P}_s)^m$,
$$\left\|\widetilde{T}(\vec q)\right\|_{\mathcal L_{p,q,s,W}}
\lesssim\|\vec q\|_{\mathcal L_{p,q,s,W}}=0,$$
which, together with Remark \ref{re74}(ii), further implies that $\widetilde{T}(\vec q)\in(\mathcal{P}_s)^m$.
From this and Lemma \ref{2501132155}, we infer that $T(\vec q)=\vec 0$,
which completes the proof of the necessity.

Now, we show the sufficiency. That is, we assume
that, for any $\gamma\in\mathbb Z_+^n$ with $|\gamma|\leq s$,
$T(x^\gamma)=0$. By Lemma \ref{2501132155}, we conclude that,
for any $\gamma\in\mathbb Z_+^n$ with $|\gamma|\leq s$,
$\widetilde{T}(x^\gamma)\in\mathcal{P}_s$.
Let $\vec f\in\mathcal L_{p,q,s,W}$, $c:=2\sqrt n$,
$z\in\mathbb{R}^n$, $t_1\in(0,\infty)$, and the cube $Q$ be
centered at $z$ with length $t\in(0,\infty)$
satisfying $Q\subset B({\bf0},t_1)=:B_1$.
For any $x\in Q$, define
\begin{align*}
E\left(\vec f,z,t\right)(x):=&\,\int_{\mathbb{R}^n\setminus  cQ}
\left[K(x,y)-\sum_{\{\gamma\in\mathbb Z^n_+:  |\gamma|\leq s\}}
\frac{\partial^\gamma_{(1)}K(z,y)}{\gamma!}(x-z)^\gamma\right]\\
&\quad\times\left[\vec f(y)-\Pi_{ cQ}^s\left(\vec f\right)\right]\,dy
\end{align*}
and
\begin{align*}
F\left(\vec f,z,t\right)(x):=&\,\int_{\mathbb{R}^n}\left[
\sum_{\{\gamma\in\mathbb Z^n_+:  |\gamma|\leq s\}}
\frac{\partial^\gamma_{(1)}K(z,y)}{\gamma!}(x-z)^\gamma
{\bf1}_{\mathbb{R}^n\setminus cQ}(y)\right.\\
&\quad\left.-\sum_{\{\gamma\in\mathbb Z^n_+:  |\gamma|\leq s\}}
\frac{\partial^\gamma_{(1)}K({\bf0},y)}{\gamma!}x^\gamma
{\bf1}_{\mathbb{R}^n\setminus cB_1}(y)\right]\left[
\vec f(y)-\Pi_{ cQ}^s\left(\vec f\right)\right]\,dy.
\end{align*}
Next, we claim that $E(\vec f,z,t)$ and $F(\vec f,z,t)$ are well-defined.
Indeed, using Definition \ref{skernel}(iv),
the fact that, for any $x\in Q$ and
$y\in\mathbb{R}^n\setminus  cQ$, $|z-y|\ge \sqrt nt>2|z-x|$,
and \eqref{xz},
we find that, for any $x\in Q$,
\begin{align}\label{btbbmc-eq3}
\left|E(\vec f,z,t)(x)\right|
&\le\int_{\mathbb{R}^n\setminus  cQ}
\left|K(x,y)-\sum_{\{\gamma\in\mathbb Z^n_+:  |\gamma|\leq s\}}
\frac{\partial^\gamma_{(1)}K(z,y)}{\gamma!}(x-z)^\gamma
\right|\left|\vec f(y)-\Pi_{ cQ}^s\left(\vec f\right)\right|\,dy\nonumber\\
&=\int_{\mathbb{R}^n\setminus  cQ}
\left|\sum_{\{\gamma\in\mathbb Z^n_+:  |\gamma|=s\}}
\frac{\partial^\gamma_{(1)}K({\xi_x},y)-
\partial^\gamma_{(1)}K(z,y)}{\gamma!}(x-z)^\gamma\right|\nonumber\\
&\quad\times\left|\vec f(y)-\Pi_{ cQ}^s\left(\vec f\right)\right|\,dy\nonumber\\
&\lesssim\int_{\mathbb{R}^n\setminus  cQ}
|x-z|^{s+\delta}\frac{|\vec f(y)-\Pi_{ cQ}^s(\vec f)|}{|y-z|^{n+s+\delta}}\,dy\nonumber\\
&\lesssim t^{n(\frac 1{p}-1)}\left\|A_Q\right\|
\left\|\vec f\right\|_{\mathcal L_{p,q,s,W}}<\infty,
\end{align}
where $\xi_x=\theta x+(1-\theta)z$ with $\theta\in [0,1]$,
which implies that $E(\vec f,z,t)$ is well-defined.
On the other hand, from the assumption that
$Q\subset B_1$, it follows that,
for any $y\in \mathbb{R}^n\setminus cB_1$,
$|z|<t_1<\frac12 |y|$ and hence
$|z-y|<\frac32|y|$. Combining this,
the assumption that $Q\subset B_1$,
and Definition \ref{skernel}(iv), we conclude that,
for any $x\in Q$ and $y\in \mathbb{R}^n\setminus cB_1$,
\begin{align}\label{btbbmc-eq1}
&\left|\sum_{\{\gamma\in\mathbb Z^n_+:  |\gamma|\leq s\}}
\frac{\partial^\gamma_{(1)}K(z,y)}{\gamma!}(x-z)^\gamma
{\bf1}_{\mathbb{R}^n\setminus cQ}(y)
-\sum_{\{\gamma\in\mathbb Z^n_+:  |\gamma|\leq s\}}
\frac{\partial^\gamma_{(1)}K({\bf0},y)}{\gamma!}x^\gamma\right|\nonumber\\
&\quad\leq\left|\sum_{\{\gamma\in\mathbb Z^n_+:  |\gamma|\leq s\}}
\frac{\partial^\gamma_{(1)}K(z,y)}{\gamma!}(x-z)^\gamma
-K(x,y)\right|\nonumber\\
&\qquad+\left|K(x,y)
-\sum_{\{\gamma\in\mathbb Z^n_+:  |\gamma|\leq s\}}
\frac{\partial^\gamma_{(1)}K({\bf0},y)}{\gamma!}x^\gamma\right|\nonumber\\
&\quad=\left|\sum_{\{\gamma\in\mathbb Z^n_+:  |\gamma|=s\}}
\frac{\partial^\gamma_{(1)}K({\xi_{x,1}},y)
-\partial^\gamma_{(1)}K(z,y)}{\gamma!}(x-z)^\gamma\right|\nonumber\\
&\quad\quad+\left|\sum_{\{\gamma\in\mathbb Z^n_+:  |\gamma|=s\}}
\frac{\partial^\gamma_{(1)}K({\xi_{x,2}},y)
-\partial^\gamma_{(1)}K({\bf0},y)}{\gamma!}x^\gamma\right|\nonumber\\
&\quad\lesssim\frac{|x-z|^{s+\delta}}{|y-z|^{n+s+\delta}}+
\frac{|x|^{s+\delta}}{|y|^{n+s+\delta}}
\sim\frac{|x-z|^{s+\delta}+|x|^{s+\delta}}{|y-z|^{n+s+\delta}}
\lesssim\frac{t_1^{s+\delta}}{|y-z|^{n+s+\delta}},
\end{align}
where $\xi_{x,1}=\theta_1 x+(1-\theta_1)z$ and
$\xi_{x,2}=\theta_2 x+(1-\theta_2)z$
with $\theta_i\in [0,1]$ for $i\in\{1,2\}$.
In addition, by Definition \ref{skernel}(i),
we obtain, for any $x\in Q$ and $y\in cB_{1}\setminus  cQ$,
\begin{align}\label{429}
&\left|\sum_{\{\gamma\in\mathbb Z^n_+:  |\gamma|\leq s\}}
\frac{\partial^\gamma_{(1)}K(z,y)}{\gamma!}(x-z)^\gamma
{\bf1}_{\mathbb{R}^n\setminus cQ}(y)
-\sum_{\{\gamma\in\mathbb Z^n_+:  |\gamma|\leq s\}}
\frac{\partial^\gamma_{(1)}K({\bf0},y)}{\gamma!}x^\gamma
{\bf1}_{\mathbb{R}^n\setminus cB_1}(y)\right|\nonumber\\
&\quad=\left|\sum_{\{\gamma\in\mathbb Z^n_+:  |\gamma|\leq s\}}
\frac{\partial^\gamma_{(1)}K(z,y)}{\gamma!}(x-z)^\gamma
\right|\nonumber\\
&\quad\lesssim\sum_{\{\gamma\in\mathbb Z^n_+:  |\gamma|\leq s\}}
\frac{|x-z|^{|\gamma|}}{|z-y|^{n+|\gamma|}}\lesssim t^{-n}.
\end{align}
Since $\vec f\in (L^q_{\rm{loc}})^m\subset (L^1_{\rm{loc}})^m$, we may
assume that, for any $x\in\mathbb R^n$ and $i\in\{1,\ldots,m\}$,
$$\left[\Pi^s_{Q}\left(\vec f\right)\right]_i(x):=\sum_{\{\gamma\in\mathbb Z^n_+:  |\gamma|\leq s\}}
c_{\gamma,i} x^\gamma.$$
Applying this, \eqref{btbbmc-eq1}, \eqref{429}, the assumption that $Q\subset B_1$ again,
\eqref{xz}, and the H\"older inequality,
we conclude that $\vec f\in (L^q_{\rm loc})^m$ and hence
\begin{align*}
\left|F\left(\vec f,z,t\right)(x)\right|
&\le\int_{\mathbb{R}^n}\left|\sum_{\{\gamma\in\mathbb Z^n_+:  |\gamma|\leq s\}}
\frac{\partial^\gamma_{(1)}K(z,y)}{\gamma!}(x-z)^\gamma
{\bf1}_{\mathbb{R}^n\setminus cQ}(y)\right.\nonumber\\
&\quad\left.-\sum_{\{\gamma\in\mathbb Z^n_+:  |\gamma|\leq s\}}
\frac{\partial^\gamma_{(1)}K({\bf0},y)}{\gamma!}x^\gamma
{\bf1}_{\mathbb{R}^n\setminus cB_1}(y)\right|
\left|\vec f(y)-\Pi_{ cQ}^s\left(\vec f\right)\right|\,dy\nonumber\\
&\lesssim
t_1^{s+\delta}\int_{\mathbb{R}^n
\setminus  cB_1}\frac{|\vec f(y)-\Pi_{ cQ}^s(\vec f)|}{|y-z|^{n+s+\delta}}
+t^{-n}\int_{ cB_1\setminus  cQ}\left|\vec f(y)-\Pi_{ cQ}^s\left(\vec f\right)\right|\,dy\nonumber\\
&\le t_1^{s+\delta}t^{-s-\delta-n(1-\frac 1{p})}\left\|A_Q\right\|
\left\|\vec f\right\|_{\mathcal L_{p,q,s,W}}\\
&\quad+t^{-n}\left[\int_{ cB_1}\left|\vec f(y)\right|\,dy+|cB_1|
\sum_{i=1}^{m}\sum_{\{\gamma\in\mathbb Z^n_+:  |\gamma|\leq s\}}
|c_{\gamma,i}| t_1^{|\gamma|}
\right]\notag\\
&\lesssim t_1^{s+\delta}t^{-s-\delta-n(1-\frac 1{p})}\left\|A_Q\right\|
\left\|\vec f\right\|_{\mathcal L_{p,q,s,W}}
+t^{-n}|cB_1|^{1-\frac1q}\left\|\vec f\right\|_{L^q(cB_1)}\\
&\quad+t^{-n}\sum_{i=1}^{m}\sum_{\{\gamma\in\mathbb Z^n_+:  |\gamma|\leq s\}}
|c_{\gamma,i}| t_1^{n+|\gamma|}\\&<\infty.
\end{align*}
Therefore, $F(\vec f,z,t)$ is also well-defined.
This, together with \eqref{btbbmc-eq3},
finishes the proof of the above claim.

Now, choose a sequence $\{t_i\}_{i\in\mathbb{N}\cap[2,\infty)}$
of $(0,\infty)$ such that, for any $i\in\mathbb{N}$,
$t_i\le t_{i+1}$ and $\lim_{i\to\infty}t_i=\infty$.
Let $T_{\{\sqrt nt_i\}_{i\in\mathbb{N}}}$ be an $(s,\delta)$-modified
Calder\'on--Zygmund operator as in Definition \ref{mcz}.
By this and the definitions of both $E(f,z,t)$ and $F(f,z,t)$,
we obtain, for any $x\in Q$,
\begin{align}\label{btbbmc-eq4}
T_{\{\sqrt nt_i\}_{i\in\mathbb{N}}}(\vec f)(x)
&=T\left(\left[\vec f-\Pi_{ cQ}^s\left(\vec f\right)\right]{\bf1}_{ cQ}\right)(x)
+E\left(\vec f,z,t\right)(x)+F\left(\vec f,z,t\right)(x)\notag\\
&\quad+T_{\{\sqrt nt_i\}_{i\in\mathbb{N}}}
\left(\Pi_{ cQ}^s\left(\vec f\right)\right)(x)\nonumber\\&\quad-
\int_{\mathbb{R}^n\setminus cB_1}\sum_{\{\gamma\in\mathbb Z^n_+:  |\gamma|\leq s\}}
\frac{\partial^\gamma_{(1)}K({\bf0},y)}{\gamma!}x^\gamma\vec f(y)\,dy.
\end{align}
Since, for any $\gamma\in\mathbb Z_+^n$ with $|\gamma|\leq s$,
$T(x^\gamma)=0$, from Lemma \ref{2501132155},
it follows that $T_{\{\sqrt nt_i\}_{i\in\mathbb{N}}}(x^\gamma)
\in\mathcal{P}_s$.
Combining this,
Theorem \ref{sczldy}(ii), \eqref{btbbmc-eq4},
and Proposition \ref{sczbouned},
we find that
\begin{align}\label{250361955}
&\inf_{\vec P\in(\mathcal{P}_s)^m}
\left\|A_Q^{-1}\left[\widetilde{T}\left(\vec f\right)-\vec P\right]\right\|_{L^q(Q)}\nonumber\\
&\quad=\inf_{\vec P\in(\mathcal{P}_s)^m}\left\|A_Q^{-1}\left[T_{\{\sqrt nt_i\}_{i\in\mathbb{N}}}
\left(\vec f\right)-\vec P\right]\right\|_{L^q(Q)}\notag\\
&\quad\le\left\|A_Q^{-1}\left[T\left(\left[\vec f-\Pi_{ cQ}^s\left(\vec f\right)\right]
{\bf1}_{ cQ}\right)+E\left(\vec f,z,t\right)\right]\right\|_{L^q(Q)}\notag\\
&\quad\le\left\|T\left(A_Q^{-1}\left[\vec f-\Pi_{ cQ}^s\left(\vec f\right)\right]
{\bf1}_{ cQ}\right)\right\|_{L^q(Q)}
+\left\|A_Q^{-1}E\left(\vec f,z,t\right)\right\|_{L^q(Q)}\notag\\
&\quad\lesssim\left\|A_Q^{-1}
\left[\vec f-\Pi_{ cQ}^s\left(\vec f\right)\right]\right\|_{L^q( cQ)}
+\left\|A_Q^{-1}E\left(\vec f,z,t\right)\right\|_{L^q(Q)}.
\end{align}
To estimate the second term in \eqref{250361955}, using Definition \ref{skernel}(iv),
\cite[Lemma 2.20]{jtyyz22i}, Lemma \ref{p01}, and the definition of $\mathcal L_{p,q,s,W}$,
we conclude that
\begin{align*}
&\int_{Q}\left|A_Q^{-1}E(\vec f,z,t)(x)\right|^q\,dx\nonumber\\
&\quad\le\int_Q\left\{\int_{\mathbb{R}^n\setminus  cQ}
\left|K(x,y)-\sum_{\{\gamma\in\mathbb Z^n_+:  |\gamma|\leq s\}}
\frac{\partial^\gamma_{(1)}K(z,y)}{\gamma!}(x-z)^\gamma\right|\right.\\
&\qquad\times\left|A_Q^{-1}
\left[\vec f(y)-\Pi_{ cQ}^s\left(\vec f\right)\right]\right|\,dy\Bigg\}^q\,dx\nonumber\\
&\quad=\int_Q\left\{\int_{\mathbb{R}^n\setminus  cQ}
\left|\sum_{\{\gamma\in\mathbb Z^n_+:  |\gamma|=s\}}
\frac{\partial^\gamma_{(1)}K({\xi_x},y)-
\partial^\gamma_{(1)}K(z,y)}{\gamma!}(x-z)^\gamma\right|\right.\\
&\qquad\times\left|A_Q^{-1}
\left[\vec f(y)-\Pi_{ cQ}^s\left(\vec f\right)\right]\right|\,dy\Bigg\}^q\,dx\nonumber\\
&\quad\lesssim\int_Q\left\{\int_{\mathbb{R}^n\setminus  cQ}
|x-z|^{s+\delta}\frac{|A_Q^{-1}[\vec f(y)-\Pi_{ cQ}^s(\vec f)]|}
{|y-z|^{n+s+\delta}}\,dy\right\}^q\,dx,\nonumber\\
&\quad\lesssim|Q|^{1+\frac{(s+\delta) q}n}\left[\sum_{k\in\mathbb N}(2^{k}t)^{-s-\delta}\left\|
A_Q^{-1}A_{2^kQ}\right\|\left\{\fint_{2^kQ}
\left|A_{2^kQ}^{-1}\left[\vec f(y)-\Pi_{2^kQ}^s
\left(\vec f\right)\right]\right|^q
\,dy\right\}^\frac{1}{q}\right]^q\nonumber\\
&\quad\lesssim|Q|^{1+\frac{(s+\delta) q}n}\left[\sum_{k\in\mathbb N}
(2^{k}t)^{-s-\delta}2^{k\widetilde{d}(1-\frac{1}{p})}\left|2^kQ\right|^{\frac1{p}-1}\right]^q
\left\|\vec f\right\|_{\mathcal L_{p,q,s,W}}\\&\quad\sim
|Q|^{1-q(1-\frac1p)}\left\|\vec f\right\|_{\mathcal L_{p,q,s,W}}^q,
\end{align*}
where $\widetilde{d}$ is the same as in Lemma \ref{p01} and
$\xi_x=\theta x+(1-\theta)z$ with $\theta\in [0,1]$.
From this, the definition of $\mathcal L_{p,q,s,W}$, and \eqref{250361955},
we infer that
\begin{align*}
\left\|\widetilde{T}\vec f\right\|_{\mathcal L_{p,q,s,W}}&=
\sup_{\mathrm{cube}\,Q}\inf_{\vec P\in(\mathcal{P}_s)^m}|Q|^{\frac1{q'}-\frac1p}
\left\{\int_{Q}\left|
A_Q^{-1}\left[\widetilde{T}\vec f(x)-\vec P(x)\right]
\right|^q\,dx\right\}^{\frac1q}\\
&\lesssim\sup_{\mathrm{cube}\,Q}|Q|^{\frac1{q'}-\frac1p}\left[
\left\|A_Q^{-1}\left(\vec f-\Pi_{ cQ}^s\left(\vec f\right)\right)\right\|_{L^q( cQ)}
+\left\|A_Q^{-1}E\left(\vec f,z,t\right)\right\|_{L^q(Q)}\right]\\
&\lesssim\left\|\vec f\right\|_{\mathcal L_{p,q,s,W}}.
\end{align*}
This finishes the proof of the sufficiency and hence Theorem \ref{btbbmc}.
\end{proof}

As an application of Theorems \ref{dual} and \ref{btbbmc},
we obtain the following theorem.

\begin{theorem}\label{CZ}
Let $p\in(0,1]$, $s\in[\lceil n(\frac{1}{p}-1)\rceil,\infty)\cap\mathbb Z_+$,
$\delta\in(0,1]$, and $W\in A_p$.
Let $T$ be an $(s,\delta)$-Calder\'on--Zygmund operator.
Then the following statements are equivalent:
\begin{itemize}
\item[{\rm(i)}]  for any $\gamma\in\mathbb Z_+^n$ with $|\gamma|\leq s$,
$T^*(x^\gamma)=0$;
\item[{\rm(ii)}]  there exists a  bounded operator
$T_:  H^p_W
\to  H^p_W$ that agrees with $T$ on all the
continuous $(p,\infty,s)_W$-atoms.
\end{itemize}
\end{theorem}

\begin{proof}
The necessity can be deduced from \cite[Theorem 5.5]{bcyy}.
Next, we show the sufficiency.
For this purpose, let
$T^*$ be the adjoint operator of $T$ as in Remark \ref{lq}(ii)
and $\widetilde{T^*}$ the corresponding $(s,\delta)$-modified
Calder\'on--Zygmund operator
as in Definition \ref{mcz}.
Using the proof of \cite[Theorem 4.4]{bcyy}, we easily find that
$H^{p,q,s}_{W,\rm{fin}}\cap\mathcal C$ is dense in $H^p_W$.
By the definitions of $H^{p,\infty,s}_{W,\rm{fin}}$
and $(p,\infty,s)_W$-atoms,
we conclude that every $\vec f\in H^{p,\infty,s}_{W,\rm{fin}}$ is
a $(p,\infty,s)_W$-atom and $\vec f\in(L^{q'}_s)^m$,
which, together with Theorem \ref{dual}, the aforementioned density, Lemma \ref{fgduiou},
the assumption that $T$ is bounded from $H^p_W$
to $H^p_W$ on all the continuous $(p,\infty,s)_W$-atoms,
and Lemma \ref{finatom}, further implies that,
for any $\vec g\in\mathcal L_{p,q,s,W}$,
\begin{align*}
\left\|\widetilde{T^*}\left(\vec g\right)\right\|_{\mathcal L_{p,q,s,W}}
&=\sup_{\genfrac{}{}{0pt}{}{\vec f\in H^{p,\infty,s}_{W,\rm{fin}}
\cap\mathcal C}
{\|\vec f\|_{H^{p,\infty,s}_{W,\rm{fin}}=1}}}
\left|\int_{\mathbb{R}^n}\vec f(x)\widetilde{T^*}(\vec g)(x)\,dx\right|\\
&=\sup_{\genfrac{}{}{0pt}{}{\vec f\in H^{p,\infty,s}_{W,\rm{fin}}
\cap\mathcal C}
{\|\vec f\|_{H^{p,\infty,s}_{W,\rm{fin}}=1}}}
\left|\int_{\mathbb{R}^n}T\left(\vec f\right)(x)\vec g(x)\,dx\right|\\
&\le\sup_{\genfrac{}{}{0pt}{}{\vec f\in H^{p,\infty,s}_{W,\rm{fin}}
\cap\mathcal C}
{\|\vec f\|_{H^{p,\infty,s}_{W,\rm{fin}}=1}}}
\left\|T\left(\vec f\right)\right\|_{H^p_W}
\left\|\vec g\right\|_{\mathcal L_{p,q,s,W}}\\
&\lesssim\sup_{\genfrac{}{}{0pt}{}{\vec f\in H^{p,\infty,s}_{W,\rm{fin}}
\cap\mathcal C}
{\|\vec f\|_{H^{p,\infty,s}_{W,\rm{fin}}=1}}}
\left\|\vec f\right\|_{H^p_W}
\left\|\vec g\right\|_{\mathcal L_{p,q,s,W}}
=\left\|\vec g\right\|_{\mathcal L_{p,q,s,W}}.
\end{align*}
This shows that $\widetilde{T^*}$ is bounded on
the $\mathcal L_{p,q,s,W}$.
Therefore, applying this and Theorem \ref{btbbmc},
we find that, for any $\gamma\in\mathbb Z_+^n$ with $|\gamma|\leq s$,
$T^*(x^\gamma)=0$.
This finishes the proof of the sufficiency
and hence Theorem \ref{CZ}.
\end{proof}

\textbf{Conflict of Interest}\quad The authors declare no conflict of interest.

\bigskip

\noindent Yiqun Chen, Dachun Yang (Corresponding author) and Wen Yuan

\medskip

\noindent Laboratory of Mathematics and Complex Systems (Ministry of Education of China),
School of Mathematical Sciences, Beijing Normal University, Beijing 100875, The
People's Republic of China

\smallskip

\noindent{\it E-mails:} \texttt{yiqunchen@mail.bnu.edu.cn} (Y. Chen)

\noindent\phantom{{\it E-mails:} }\texttt{dcyang@bnu.edu.cn} (D. Yang)

\noindent\phantom{{\it E-mails:} }\texttt{wenyuan@bnu.edu.cn} (W. Yuan)


\begin{thebibliography}{10}

\bibitem{bgx25}

T. Bai, P. Guo and J. Xu,
On matrix weighted Bourgain--Morrey Triebel--Lizorkin spaces,
arXiv: 2508.11981.

\vspace{-0.3cm}

\bibitem{bx24-1}
T. Bai and J. Xu,
Pseudo-differential operators on matrix-weighted
Besov--Triebel--Lizorkin spaces,
Bull. Iranian Math. Soc. 50 (2024), Paper No. 31, 26 pp.

\vspace{-0.3cm}

\bibitem{bx24-2}
T. Bai and J. Xu,
Non-regular pseudo-differential operators on
matrix-weighted Besov--Triebel--Lizorkin spaces,
J. Math. Study 57 (2024), 84--100.

\vspace{-0.3cm}

\bibitem{bx24-3}
T. Bai and J. Xu,
Precompactness in matrix-weighted Bourgain--Morrey spaces,
Filomat (to appear) or arXiv: 2406.11531.

\vspace{-0.3cm}

\bibitem{b03}
M. Bownik, Anisotropic Hardy spaces and wavelets,
Mem. Amer. Math. Soc. 164 (2003), no. 781, vi+122 pp.

\vspace{-0.3cm}

\bibitem{bc22}
M. Bownik and D. Cruz-Uribe,
Extrapolation and factorization of matrix weights,
arXiv: 2210.09443v3.

\vspace{-0.3cm}

\bibitem{bcyy}
F. Bu, Y. Chen, D. Yang and W. Yuan,
Maximal function and atomic characterizations of matrix-weighted
Hardy spaces with their applications to boundedness of
Calder\'on--Zygmund Operators, Submitted or
arXiv: 2501.18800.

\vspace{-0.3cm}

\bibitem{bhyy1}
F. Bu, T. Hyt\"{o}nen, D. Yang and W. Yuan,
Matrix-weighted Besov-type and Triebel--Lizorkin-type spaces I:
$A_p$-dimensions of matrix weights and $\varphi$-transform characterizations,
Math. Ann. 391 (2025), 6105--6185.

\vspace{-0.3cm}

\bibitem{bhyy2}
F. Bu, T. Hyt\"{o}nen, D. Yang and W. Yuan,
Matrix-weighted Besov-type and Triebel--Lizorkin-type spaces II:
Sharp boundedness of almost diagonal operators,
J. Lond. Math. Soc. (2) 111 (2025), Paper No. e70094, 59 pp.

\vspace{-0.3cm}

\bibitem{bhyy3}
F. Bu, T. Hyt\"{o}nen, D. Yang and W. Yuan,
Matrix-weighted Besov-type and Triebel--Lizorkin-type spaces III:
Characterizations of molecules and wavelets, trace theorems,
and boundedness of pseudo-differential operators and
Calder\'on--Zygmund operators, Math. Z. 308 (2024),
Paper No. 32, 67 pp.

\vspace{-0.3cm}

\bibitem{bhyyNew}
F. Bu, T. Hyt\"{o}nen, D. Yang and W. Yuan,
Besov--Triebel--Lizorkin-type spaces with matrix $A_{\infty}$ weights,
Sci. China Math. (2025), https://doi.org/10.1007/s11425-024-2385-x.

\vspace{-0.3cm}

\bibitem{byymz}
F. Bu, D. Yang and W. Yuan,
Real-variable characterizations and their applications of
matrix-weighted Besov spaces on spaces of homogeneous type,
Math. Z. 305 (2023), Paper No. 16, 81 pp.

\vspace{-0.3cm}

\bibitem{byyzhang}
F. Bu, D. Yang, W. Yuan and M. Zhang,
Matrix-weighted Besov--Triebel--Lizorkin spaces of optimal scale:
real-variable characterizations,
invariance on integrable index, and Sobolev-type embedding,
Submitted or arXiv: 2505.02136.

\vspace{-0.3cm}

\bibitem{byyz}
F. Bu, D. Yang, W. Yuan and Y. Zhao,
Matrix weights, maximal operators, Calder\'on--Zygmund operators,
and Besov--Triebel--Lizorkin-type spaces --- A survey,
Anal. Theory Appl. (to appear).

\vspace{-0.3cm}

\bibitem{C64}
S. Campanato, Propriet\`a di una famiglia di spazi funzionali,
Ann. Scuola Norm. Sup. Pisa
Cl. Sci. (3) 18 (1964), 137--160.

\vspace{-0.3cm}

\bibitem{C66}
S. Campanato, Su un teorema di interpolazione di G. Stampacchia, Ann. Scuola Norm. Sup.
Pisa Cl. Sci. (3) 20 (1966), 649--652.

\vspace{-0.3cm}

\bibitem{CW21}
D. Chae and J. Wolf, The Euler equations in a critical case of the generalized Campanato
space, Ann. Inst. H. Poincar\'e C Anal. Non Lin\'eaire 38 (2021), 201--241.

\vspace{-0.3cm}

\bibitem{CW23}
D. Chae and J. Wolf, Transport equation in generalized Campanato spaces, Rev. Mat.
Iberoam. 39 (2023), 1725--1770.

\vspace{-0.3cm}

\bibitem{cg01}
M. Christ and M. Goldberg,
Vector $A_2$ weights and a Hardy--Littlewood maximal function,
Trans. Amer. Math. Soc. 353 (2001), 1995--2002.

\vspace{-0.3cm}

\bibitem{dhl20}
F. Di Plinio, T. Hyt\"{o}nen and K. Li,
Sparse bounds for maximal rough singular
integrals via the Fourier transform,
Ann. Inst. Fourier (Grenoble) 70 (2020), 1871--1902.

\vspace{-.3cm}

\bibitem{cm78}
R. R. Coifman and Y. Meyer,
Au del\`{a} des op\'erateurs pseudo-diff\'erentiels,
Ast\'erisque 57 (1978), i+185 pp.

\vspace{-.3cm}

\bibitem{c25}
D. Cruz-Uribe,
Recent results on matrix weighted norm inequalities,
arXiv: 2508.13352.

\vspace{-0.3cm}

\bibitem{dptv24}
K. Domelevo, S. Petermichl, S. Treil and A. Volberg,
The matrix $A_2$ conjecture fails, i.e. $3/2>1$,
arXiv: 2402.06961.

\vspace{-0.3cm}

\bibitem{DL19}
K. Du and J. Liu, On the Cauchy problem for stochastic parabolic equations in Holder spaces,
Trans. Amer. Math. Soc. 371 (2019), 2643--2664.

\vspace{-0.3cm}

\bibitem{Duo01}
J. Duoandikoetxea, Fourier Analysis, Graduate Studies
in Mathematics 29, American
Mathematical Society, Providence, RI, 2001.

\vspace{-0.3cm}

\bibitem{dly21} X. T. Duong, J. Li and D. Yang, Variation of
Calder\'on--Zygmund operators with matrix weight, Commun. Contemp.
Math. 23 (2021), Paper No. 2050062, 30 pp.

\vspace{-0.3cm}

\bibitem{FS} C. Fefferman and E. M. Stein, $H^p$ spaces of several variables,
Acta Math. 129 (1972) 137--193.


\vspace{-0.3cm}

\bibitem{fr04}
M. Frazier and S. Roudenko, Matrix-weighted
Besov spaces and conditions of $A_p$ type for
$0<p\le1$, Indiana Univ. Math. J. 53 (2004), 1225--1254.

\vspace{-0.3cm}

\bibitem{fr21}
M. Frazier and S. Roudenko,
Littlewood--Paley theory for matrix-weighted function spaces,
Math. Ann. 380 (2021), 487--537.

\vspace{-0.3cm}

\bibitem{G79}
J. Garc\'{\i}a-Cuerva, Weighted $H^p$ spaces,
Dissertationes Math. 162 (1979), 63 pp.

\vspace{-0.3cm}

\bibitem{GR} J.  Garc\'{\i}a-Cuerva and J. Rubio de Francia,
Weighted Norm Inequalities and Related Topics,
North-Holland Mathematics Studies  116, North-Holland Publishing Co., Amsterdam, 1985.

\vspace{-0.3cm}

\bibitem{Gold}
M. Goldberg, Matrix $A_p$ weights via maximal functions,
Pacific J. Math. 211 (2003), 201--220.

\vspace{-0.3cm}

\bibitem{h08}
N. J. Higham,
Functions of Matrices. Theory and Computation,
Society for Industrial and Applied Mathematics, Philadelphia, PA, 2008.

\vspace{-0.3cm}

\bibitem{ho19}
K.-P. Ho,
Integral operators on BMO and Campanato spaces,
Indag. Math. (N.S.) 30 (2019), 1023--1035.

\vspace{-0.3cm}

\bibitem{hj94}
R. A. Horn and C. R. Johnson,
Topics in Matrix Analysis,
Corrected reprint of the 1991 original,
Cambridge University Press, Cambridge, 1994.

\vspace{-0.3cm}

\bibitem{hj13}
R. A. Horn and C. R. Johnson,
Matrix Analysis, Second edition,
Cambridge University Press, Cambridge, 2013.

\vspace{-0.3cm}

\bibitem{HZ19}
X. Hu and J. Zhou, An inequality in weighted Campanato spaces
with applications. Anal. Math. 45 (2019), 515--526.

\vspace{-0.3cm}

\bibitem{HBLWY}
W. Hu, J. J. Betancor, S. Liu, H. Wu and D. Yang,
Boundedness of operators on weighted Morrey--Campanato spaces
in the Bessel setting, J. Geom. Anal. 34 (2024), Paper No. 72, 40 pp.

\vspace{-.3cm}

\bibitem{jtyyz22i}
H. Jia, J. Tao, D. Yang, W. Yuan and Y. Zhang, Boundedness of Calder\'on--Zygmund operators on
special John--Nirenberg--Campanato and Hardy-type spaces via congruent cubes,
Anal. Math. Phys. 12 (2022), Paper No. 15, 56 pp. 		
\vspace{-0.3cm}	

\bibitem{jtyyz22iii}
H. Jia, J. Tao, D. Yang, W. Yuan and Y. Zhang, Special John--Nirenberg--Campanato spaces via congruent cubes, Sci. China Math. 65 (2022), 359--420.			

\vspace{-.3cm}

\bibitem{jly}
Y. Jin, Y. Li and D. Yang,
Besov--Bourgain--Morrey--spaces:
boundedness of operators, duality, and sharp
John--Nirenberg inequality,
Anal. Math. Phys. 15 (2025), Paper No. 104, 70 pp.

\vspace{-.3cm}

\bibitem{JN61}
F. John and L. Nirenberg, On functions of bounded mean oscillation, Comm. Pure Appl.
Math. 14 (1961), 415--426.

\vspace{-.3cm}

\bibitem{kk13}
S. Kislyakov and N. Kruglyak,
Extremal Problems in Interpolation Theory, Whitney--Besicovitch Coverings,
and Singular Integrals,
IMPAN Monogr. Mat. (N. S.) 74,
Birk-h\"auser/Springer Basel AG, Basel, 2013.

\vspace{-0.3cm}

\bibitem{llor23}
A. K. Lerner, K. Li, S. Ombrosi and I. P. Rivera-R\'\i os,
On the sharpness of some quantitative
Muckenhoupt--Wheeden inequalities,
C. R. Math. Acad. Sci. Paris 362 (2024), 1253--1260.

\vspace{-0.3cm}

\bibitem{llor24}
A. K. Lerner, K. Li, S. Ombrosi and I. P. Rivera-R\'{\i}os,
On some improved weighted weak type inequalities,
Ann. Sc. Norm. Super. Pisa Cl. Sci. (5),
https://doi.org/10.2422/2036-2145.202407\_012.

\vspace{-0.3cm}

\bibitem{lyy24-1}
Z. Li, D. Yang and W. Yuan,
Matrix-weighted Besov--Triebel--Lizorkin spaces
with logarithmic smoothness,
Bull. Sci. Math. 193 (2024), Paper No. 103445, 54 pp.

\vspace{-0.3cm}

\bibitem{lyy24-2}
Z. Li, D. Yang and W. Yuan,
Matrix-weighted Poincar\'{e}-type inequalities
with applications to logarithmic
Haj\l asz--Besov spaces on spaces of homogeneous type,
Submitted.

\vspace{-0.3cm}

\bibitem{MW75}
B. Muckenhoupt and R. Wheeden,
Weighted bounded mean oscillation and the Hilbert transform,
Studia Math. 54 (1975/76), 221--237.

\vspace{-0.3cm}

\bibitem{N10}
E. Nakai, Singular and fractional integral operators on Campanato spaces with variable
growth conditions, Rev. Mat. Complut. 23 (2010), 355--381.

\vspace{-0.3cm}

\bibitem{N17}
E. Nakai, Singular and fractional integral operators on preduals of Campanato spaces with
variable growth condition, Sci. China Math. 60 (2017), 2219--2240.

\vspace{-0.3cm}

\bibitem{ns2012}
E. Nakai and Y. Sawano,
Hardy spaces with variable exponents and generalized
Campanato spaces, J. Funct. Anal. 262 (2012), 3665--3748.

\vspace{-0.3cm}

\bibitem{NY19}
E. Nakai and T. Yoneda, Applications of Campanato spaces with variable growth condition
to the Navier--Stokes equation, Hokkaido Math. J. 48 (2019), 99--140.

\vspace{-0.3cm}

\bibitem{nptv17}
F. Nazarov, S. Petermichl, S. Treil and A. Volberg,
Convex body domination and weighted estimates with matrix weights,
Adv. Math. 318 (2017), 279--306.

\vspace{-0.3cm}

\bibitem{nt96}
F. L. Nazarov and S. R. Tre\u{\i}l',
The hunt for a Bellman function: applications to estimates for
singular integral operators and to other classical problems of harmonic analysis,
(Russian) Algebra i Analiz 8(5) (1996), 32--162;
translation in St. Petersburg Math. J. 8 (1997), 721--824.
\vspace{-0.3cm}

\bibitem{n12} M. Nielsen, On transference of multipliers on
matrix weighted $L^p$-spaces, J. Geom. Anal. 22 (2012),
12--22.

\vspace{-0.3cm}

\bibitem{n25} M. Nielsen, Matrix weighted $\alpha$-modulation spaces,
Monatsh. Math. 206 (2025), 419--448.

\vspace{-0.3cm}

\bibitem{n25-2} M. Nielsen, Bandlimited multipliers on
matrix-weighted $L^p$-spaces, J. Fourier Anal. Appl. 31 (2025),
Paper No. 3, 10 pp.

\vspace{-0.3cm}

\bibitem{nr18} M. Nielsen and M. G. Rasmussen, Projection operators on matrix
weighted $L^p$ and a simple sufficient Muckenhoupt condition,
Math. Scand. 123 (2018), 72--84.

\vspace{-0.3cm}

\bibitem{rou03}
S. Roudenko,
Matrix-weighted Besov spaces,
Trans. Amer. Math. Soc. 355 (2003), 273--314.

\vspace{-0.3cm}

\bibitem{rou04}
S. Roudenko,
Duality of matrix-weighted Besov spaces,
Studia Math. 160 (2004), 129--156.

\vspace{-0.3cm}

\bibitem{S93}
E. M. Stein,
Harmonic Analysis: Real-Variable Methods, Orthogonality, and Oscillatory Integrals,
Princeton Univ. Press, Princeton, NJ, 1993.

\vspace{-0.3cm}

\bibitem{TW80}
M. H. Taibleson and G. Weiss,
The molecular characterization of certain Hardy spaces,
Representation theorems for Hardy spaces, pp. 67--149,
Astérisque, 77, Soc. Math. France, Paris, 1980.

\vspace{-0.3cm}

\bibitem{TYY19}
J. Tao, D. Yang and W. Yuan, John--Nirenberg--Campanato spaces, Nonlinear Anal. 189
(2019), Paper No. 111584, 36 pp.

\vspace{-0.3cm}

\bibitem{tv97}
S. Treil and A. Volberg,
Wavelets and the angle between past and future,
J. Funct. Anal. 143 (1997), 269--308.

\vspace{-0.3cm}

\bibitem{v97}
A. Volberg,
Matrix $A_p$ weights via $S$-functions,
J. Amer. Math. Soc. 10 (1997), 445--466.

\vspace{-0.3cm}

\bibitem{v24}
E. Vuorinen, The strong matrix weighted maximal operator,
Adv. Math. 453 (2024), Paper No. 109847, 18 pp.

\vspace{-0.3cm}

\bibitem{wgx25}
S. Wang, P. Guo and J. Xu,
Embedding and Duality of Matrix-weighted Modulation Spaces,
Taiwanese J. Math. 29 (2025), 171--187.

\vspace{-0.3cm}

\bibitem{wgx25b}
S. Wang, P. Guo and J. Xu,
Precompact sets in matrix weighted Lebesgue spaces with variable exponent,
Georgian Math. J. (2025), https://doi.org/10.1515/gmj-2025-2029.

\vspace{-0.3cm}

\bibitem{wyy23}
Q. Wang, D. Yang and Y. Zhang, Real-variable characterizations and their
applications of matrix-weighted Triebel--Lizorkin spaces,
J. Math. Anal. Appl. 529 (2024), Paper No. 127629, 37 pp.

\vspace{-0.3cm}

\bibitem{wm58}
N. Wiener and P. Masani,
The prediction theory of multivariate stochastic processes. II. The linear predictor,
Acta Math. 99 (1958), 93--137.

\vspace{-0.3cm}

\bibitem{YN21}
S. Yamaguchi and E. Nakai, Generalized fractional integral operators on Campanato spaces
and their bi-preduals, Math. J. Ibaraki Univ. 53 (2021), 17--34.

\vspace{-0.3cm}

\bibitem{YNS25}
S. Yamaguchi, E. Nakai and K. Shimomura, Bi-predual spaces of generalized Campanato
spaces with variable growth condition, Acta Math. Sin. (Engl. Ser.) 41 (2025), 273--303.

\vspace{-0.3cm}

\bibitem{yyy}
X. Yan, D. Yang and W. Yuan,
Intrinsic square function characterizations of Hardy spaces
associated with ball quasi-Banach function spaces,
Front. Math. China 15 (2020), 769--806.

\vspace{-0.3cm}

\bibitem{YY10}
D. Yang and S. Yang,
Elementary characterizations of generalized weighted
Morrey--Campanato spaces. Appl. Math. J. Chinese Univ. Ser. B 25
(2010), 162--176.

\vspace{-0.3cm}

\bibitem{YY11}
D. Yang and S. Yang,
New characterizations of weighted Morrey--Campanato spaces,
Taiwanese J. Math. 15 (2011), 141--163.

\vspace{-0.3cm}

\bibitem{YYZhang}
D. Yang, W. Yuan and M. Zhang,
Matrix-weighted Besov--Triebel--Lizorkin
spaces of optimal scale:
boundedness of pseudo-differential, trace, and Calderón--Zygmund
operators, Proc. Steklov Inst. Math.
(to appear) or arXiv: 2504.19060.

\vspace{-0.3cm}

\bibitem{ZCTY23}
Z. Zeng, D. Chang, J. Tao and D. Yang, Nontriviality of John--Nirenberg--Campanato spaces,
Complex Anal. Oper. Theory 17 (2023), Paper No. 70, 47 pp.

\end{thebibliography}
\end{document}